\documentclass[reqno, xcolor=pdftex,a4paper]{amsart}

\usepackage[english]{babel}
\usepackage{amsmath,amsthm,amssymb,amsfonts}
\usepackage{mathrsfs}
\usepackage[utf8x]{inputenc}

\usepackage{enumitem}
\setlist[enumerate]{leftmargin=*}

\usepackage{hyperref}
\usepackage[utf8x]{inputenc}

\newcommand{\e}{\mathrm{e}}

\usepackage[non-compressed-cites,initials,nobysame]{amsrefs}

\renewcommand{\approx}{ \asymp}

\def\bX{{\mathbf X}}

\newcommand{\Z}{\mathbb{Z}}

\def\cU{{\mathcal U}}

\def\sC{{\mathscr C}}

\DeclareMathOperator{\supp}{supp}
\DeclareMathOperator{\oD}{D}
\DeclareMathOperator{\oG}{G}
\DeclareMathOperator{\oS}{S}

\DeclareFontFamily{U}{mathx}{\hyphenchar\font45}
\DeclareFontShape{U}{mathx}{m}{n}{
      <5> <6> <7> <8> <9> <10>
      <10.95> <12> <14.4> <17.28> <20.74> <24.88>
      mathx10
      }{}
\DeclareSymbolFont{mathx}{U}{mathx}{m}{n}
\DeclareFontSubstitution{U}{mathx}{m}{n}
\DeclareMathAccent{\widecheck}{0}{mathx}{"71}
\DeclareMathAccent{\wideparen}{0}{mathx}{"75}

\makeatletter
\newcommand{\leqnomode}{\tagsleft@true}
\newcommand{\reqnomode}{\tagsleft@false}
\makeatother

\numberwithin{equation}{section}

\usepackage[usenames,dvipsnames]{color}

\newcommand{\dd}{d}

\newcommand{\R}{\mathbb{R}}
\newcommand{\N}{\mathbb{N}}

\newcommand{\Ls}{\mathcal{L}}
\newcommand{\As}{\mathcal{A}}

\theoremstyle{theorem}
\newtheorem{theorem}{\sc \textbf{Theorem}}[section]  
\newtheorem{proposition}[theorem]{\sc \textbf{Proposition}}   
\newtheorem{corollary}[theorem]{\sc \textbf{Corollary}}        
\newtheorem{lemma}[theorem]{\sc \textbf{Lemma}}

\renewcommand{\approx}{ \asymp}

\theoremstyle{remark}
\newtheorem{definition}[theorem]{\sc \textbf{Definition}}

\newtheorem{remark}[theorem]{\sc \textbf{Remark}}

\usepackage[T1]{fontenc}
\DeclareFontFamily{T1}{calligra}{}
\DeclareFontShape{T1}{calligra}{m}{n}{<->s*[1.44]callig15}{}
\DeclareMathAlphabet\mathcalligra   {T1}{calligra} {m} {n}
\DeclareMathAlphabet\mathzapf       {T1}{pzc} {mb} {it}
\DeclareMathAlphabet\mathchorus     {T1}{qzc} {m} {n}
\DeclareMathAlphabet\mathrsfso      {U}{rsfso}{m}{n}

\newcommand{\unif}{{\rm unif}}
\begin{document}

\title[Multipliers for Triebel--Lizorkin and Besov spaces]{Pointwise multipliers for \\ Triebel--Lizorkin and Besov spaces on Lie groups} 

\author[T.\ Bruno]{Tommaso Bruno}
\address{Dipartimento di Matematica, Universit\`a degli Studi di Genova\\ Via Dodecaneso 35, 16146 Genova, Italy}
\email{brunot@dima.unige.it}

\author[M.\ M.\ Peloso]{Marco M.\ Peloso}
\address{Dipartimento di Matematica, 
Universit\`a degli Studi di Milano, 
Via C.\ Saldini 50,  
20133 Milano, Italy - Dipartimento di Eccellenza 2023-2027}
\email{marco.peloso@unimi.it}

\author[M.\ Vallarino]{Maria Vallarino}
\address{Dipartimento di Scienze Matematiche ``Giuseppe Luigi Lagrange'',
  Politecnico di Torino, Corso Duca degli Abruzzi 24, 10129 Torino,
  Italy }
\email{maria.vallarino@polito.it}

\keywords{Lie groups, pointwise multipliers, Besov spaces, Triebel--Lizorkin spaces}
\thanks{{\em{Math Subject Classification}} 46E35, 22E30, 43A15} 
\thanks{The first and third authors are partially supported by the  2022 INdAM--GNAMPA grant {\em Generalized Laplacians on continuous and discrete structures} (CUP\_E55F22000270001). The second author is partially supported by the 2022 INdAM--GNAMPA grant {\em Holomorphic Functions in One and Several Complex Variables} (CUP\_E55F22000270001). All authors are members of the Gruppo Nazionale per l'Analisi Matematica, la Probabilit\`a e le loro Applicazioni (GNAMPA) of the Istituto Nazionale di Alta Matematica (INdAM)} 

\begin{abstract}
On a general Lie group $G$ endowed with a sub-Riemannian structure and of local dimension $d$, we characterize the pointwise multipliers of Triebel--Lizorkin spaces $F^{p,q}_{\alpha}$ for $p,q\in (1,\infty)$ and $\alpha>d/p$, and those of Besov spaces $B^{p,q}_{\alpha}$ for $q\in [1,\infty]$, $p>d$ and $d/p< \alpha<1$. When $G$ is stratified, we extend the latter characterization to all $p,q\in [1,\infty]$ and  $\alpha>d/p$.
\end{abstract}
\maketitle

\section{Introduction}
The problem of describing explicitly the pointwise multipliers of function spaces is one of the basic questions when studying their role, in particular, in the theory of partial differential equations. In the Euclidean setting, the case of Sobolev spaces was first consider by Strichartz~\cite{Str-67}; his result was then extended to the case of Triebel--Lizorkin spaces by a number of authors, see e.g.~\cite[2.8]{TriebelTFS} and the references therein. The case of Besov spaces turned out to be more difficult and was object of several attempts, see e.g.~\cite{Peetre,MS,Sickel,SS}, until it was very recently solved by Nguyen and Sickel~\cite{NS}. To the best of our knowledge, however, no result is available in higher generality than $\R^{d}$. In this paper we consider such problem in the case of Besov and Triebel--Lizorkin spaces defined in the sub-elliptic setting of a general Lie group.

Beyond the classical potential spaces on $\R^d$, in recent years the theory of function spaces on manifolds, in particular when these are endowed with a sub-Riemannian structure, has been at the center of intense research efforts. The standard prototype for such a situation is the case of connected Lie groups, when the classical Laplacian is replaced by the intrinsic  sub-Laplacian with respect to a H\"ormander system $\mathbf X$ of left-invariant vector fields.  The ground work for Sobolev, Triebel--Lizorkin and Besov spaces on general Lie groups was laid in~\cite{BPTV,BPV-MA,BPV-GAHA}, see also~\cite{CRTN,Fe, GS}, where equivalent descriptions and norms, embeddings, interpolation and algebra properties, among other things, were obtained.

\smallskip

The aim of this paper is then to characterize the pointwise multipliers for such Triebel--Lizorkin spaces $F^{p,q}_\alpha$  and Besov spaces $B^{p,q}_\alpha$, which we denote by $MF^{p,q}_\alpha$  and  $MB^{p,q}_\alpha$ respectively, on a noncompact connected Lie group $G$. We obtain a complete characterization of $MF^{p,q}_\alpha$  in the range $1<p,q<\infty$, and $\alpha>d/p$, where $d$ is the so-called ``local dimension'' of $G$, which depends only on $G$ and $\bf X$. The case of Besov spaces  turns out to be more challenging and, to a certain extent, this should not come as a surprise in view of the Euclidean case already. For $MB^{p,q}_\alpha$ we obtain a complete characterization in the case when $G$ is a {\em stratified} Lie group. In the case of a general Lie group, we characterize the multiplier space $MB^{p,q}_\alpha$ only for certain ranges of the smoothness parameter $\alpha$ and for $p>d$. 

The reason of this restriction is merely technical, and is due to our use (inspired by~\cite{NS}) of an equivalent Besov norm expressed in terms of a finite difference. Interestingly, it seems not clear what a satisfactory definition of a finite difference of arbitrary high order should be on a general Lie group. We are able to say this on a stratified group, and we discuss the general case at the very end of the paper.

\smallskip

The structure of the paper is as follows. In Section~\ref{2nd-sec} we introduce some preliminaries about the sub-elliptic setting of a Lie group; in Sections~\ref{Sec:TL} and~\ref{Sec:Besov} we characterize $MF^{p,q}_{\alpha}$ and $MB^{p,q}_{\alpha}$ for the above mentioned indices when $G$ is a general Lie group $G$, while in the final Section~\ref{Sec:stratified} we extend the characterization of $MB^{p,q}_{\alpha}$ to all regularities when $G$ is stratified.

\section{Setting and preliminaries}\label{2nd-sec}
Let $G$ be a noncompact connected Lie group with identity $e$, let $\lambda$ be a left Haar measure on $G$ and $\delta$ be the modular function.  We pick a family $\mathbf{X}= \{ X_1,\dots, X_\kappa\}$ of left-invariant linearly independent vector fields which satisfy H\"ormander's condition, and denote by $d_C$ the associated left-invariant Carnot--Carath\'edory distance. We shall sometimes write $|x|=d_C(x,e)$, and denote by $B_r$ the ball centered at $e$ of radius $r$. We recall that the metric measure space $(G, d_C,\lambda)$ is locally doubling, as there exists $d\in \N$ (depending on $G$ and $\bf{X}$) such that  
\[
C^{-1} r^d \leq \lambda(B_{r}) \leq C r^d\quad \forall r\in (0,1],
\]
where $C>0$ is a constant independent of $r$; but that in general it is not doubling, as the growth of $\lambda(B_{r})$ can be exponential for large $r$'s. For this fact and all what follows, we refer the reader to~\cite{BPV-GAHA, BPV-MA, BPV-NA, BPTV} and the references therein.

If $p\in [1,\infty]$, we shall denote by $L^{p}$ the classical Lebesgue spaces with respect to $\lambda$, and their norms will be denoted by $\| \cdot \|_{p}$. The convolution between two functions $f$ and $g$, when it exists, is defined by 
\[
f*g(x) =\int_G f(xy)g(y^{-1})\, \dd \lambda(y),\qquad  x\in G.
\]
We denote by $\Ls$ the operator
\begin{equation*}
\Ls = - \sum_{j=1}^\kappa (X_j^2 +(X_j\delta)(e) X_j),
\end{equation*}
which is symmetric on $L^2$, is essentially self-adjoint on
$C_c^\infty(G)$, and is the intrinsic sub-Laplacian associated with
$\bX$; see~\cite{Agrachev-et-al, HMM}. We shall denote by $\Ls$ its unique self-adjoint extension too.

The operator $\Ls$ is the infinitesimal generator of the diffusion (heat) semigroup $(\e^{-t\Ls})_{t>0}$, which has a smooth convolution kernel which we denote by $p_{t}$, $t>0$. It is well known, cf.\ e.g.~\cite[Lemma 3.1]{BPV-MA}, that there exist constants $C,c_{1},c_{2}>0$ such that
\begin{equation}\label{heatabovebelow}
C^{-1}t^{-\frac{d}{2}} \e^{-c_{1} \frac{|x|^{2}}{t}}  \leq  p_{t}(x)\leq C\,  t^{-\frac{d}{2}} \e^{-c_{2} \frac{|x|^{2}}{t}}, \qquad \forall t\in(0,1),\,x\in G, 
\end{equation}
and that for all $h\in \N$ there exist positive constants $C=C(h)$ and $b=b_h$ such that
\begin{equation}\label{heatestimate}
|X_J p_t (x)| \leq C t^{-\frac h2} p_{bt}(x) \qquad \forall t\in(0,1),\,x\in G,\, J\in \{1, \dots, \kappa \}^h.
\end{equation}
Here and all throughout, for $J = (J_{1}, \dots, J_{h})\in \{1, \dots, \kappa \}^h$ the notation $X_J $ stands for the differential operator $X_{J_{1}} \cdots X_{J_{h}}$.

\subsection{Triebel--Lizorkin and Besov spaces for $\Ls$} Suppose $\alpha>0$ and $q\in [1,\infty]$. For $p\in [1,\infty)$, the Triebel--Lizorkin space $F_\alpha^{p,q}$ is the space of functions $f\in L^{p}$ such that, when $m$ is the smallest integer larger than $\alpha/2$,
\begin{equation}\label{Fnorm}
\|f\|_{F_\alpha^{p,q}} = \|f\|_{p} + \bigg\| \bigg( \int_0^1 ( t^{-\alpha/2} |(t\Ls)^{m} \e^{-t\Ls}  f| )^q \, \frac{\dd t}{t}\bigg)^{1/q} \bigg\|_{p}
\end{equation}
is finite, with the usual modification when $q=\infty$. For $ p\in [1,\infty]$, the Besov space $B_\alpha^{p,q}$ is the space of functions $f\in L^{p}$ such that, when $m$ is as above,
\begin{equation}\label{Bnorm}
\|f\|_{B_\alpha^{p,q}} = \|f\|_{p} + \bigg( \int_0^1 (t^{-\alpha/2}\,  \| (t\Ls)^{m} \e^{-t\Ls}    f\|_{p})^q \, \frac{\dd t}{t}\bigg)^{1/q}
\end{equation}
is finite, again with the usual modification when $q=\infty$. By~\cite[Theorem 4.1]{BPV-MA}, for the above $p$ and $q$'s any other choice of an integer $m> \alpha/2$ in~\eqref{Fnorm} and~\eqref{Bnorm} gives (respectively) equivalent norms. In case no distinction between $F^{p,q}_{\alpha}$ and $B^{p,q}_{\alpha}$ is needed, we shall write $X^{p,q}_{\alpha}$ to denote either of the two. We recall that, by~\cite[Theorem 5.2]{BPV-MA}, if $p\in (1,\infty)$ and $\alpha>0$, then the space $F^{p,2}_{\alpha}$ coincides with the Sobolev space $L^{p}_{\alpha}$ (cf.~\cite{BPTV}), namely the space of functions $f\in L^{p}$ whose norm
\[
\| f\|_{L^{p}_{\alpha}} = \|f\|_{p} + \|\Ls^{\alpha/2}f\|_{p}
\]
is finite. If $\alpha=k\in \N$, moreover, by~\cite[Proposition 3.3]{BPTV}
\begin{equation}\label{Sobinteger}
\| f\|_{L^{p}_{k}}  \approx \sum_{0\leq |J|\leq k}\|X_{J}f\|_{p}.
\end{equation}
For later convenience, we define when $p=\infty$
\begin{equation}\label{Sobinfty}
\| f\|_{L^{\infty}_{k}}  = \sum_{0\leq |J|\leq k}\|X_{J}f\|_{\infty}.
\end{equation}

Here and in what follows, $A \approx B$ for two positive quantities $A$ and $B$ means that there exists $C>0$ (depending on $G$ and other circumstantial parameters) such that $C^{-1} B \leq A \leq  C\, B$. Analogously, we shall write $A\lesssim B$ if there exists such a $C$ such that $A\leq C \,B$.

We finally recall that, given $p,q\in [1,\infty]$ and $\alpha>d/p$, the spaces $B^{p,1}_{d/p}$ and $B^{p,q}_{\alpha}$ are algebras under pointwise multiplication; and that the same holds for the spaces $F^{p,q}_{\alpha}$, provided $p\in (1,\infty)$. See~\cite[Theorem 7.1]{BPV-MA}. In particular, if $f,g\in X^{p,q}_{\alpha}$ and the indices are as above, then
\begin{equation}\label{algebra}
\| f g\|_{X^{p,q}_\alpha}\lesssim  \|f\|_{X^{p,q}_\alpha} \|g\|_{X^{p,q}_\alpha}.
\end{equation}
\subsection{First order finite differences and equivalent norms} We introduce now first-order finite differences on $G$, and recall their role in providing equivalent norms for the spaces $X^{p,q}_{\alpha}$ . Higher order differences will be discussed in due course, see in particular Section~\ref{Sec:stratified} and Remark~\ref{remarkHOD}. 

For $y\in G$, we define the first-order difference $\oD_{y}$ of a function $f$ as
\begin{equation}\label{oD}
 \oD_y f(x) =f(xy^{-1})- f(x), \qquad x\in G.
 \end{equation}
If $q\in [1,\infty]$ and $\alpha\in(0,1)$, we consider the associated functionals (to lighten the notation, we write $V(u) = \lambda(B_{u})$ for $u>0$)
\[
\mathcal{S}^{{\rm loc}, q}_{\alpha}f(x) =\bigg(\int_0^1\bigg[ \frac{1}{u^{\alpha}  V(u)}\int_{|y|<u}|\oD_y f(x)|\, \dd\lambda(y) \bigg]^q\,\frac{\dd u}{u}\bigg)^{1/q}, \qquad x\in G,
\]
and, if also $p\in [1,\infty]$,
\[
\mathcal{A}^{p,q}_\alpha(f)= 
\bigg(\int_{|y|\leq 1} \bigg(\frac{\|\oD_y  f\|_{p}}{|y|^\alpha}\bigg) ^q \, \frac{\dd\lambda(y)}{V(|y|)}\bigg)^{1/q}.
\]
By~\cite[Theorem 8]{BPV-GAHA}, if $p,q\in (1,\infty)$ and $\alpha \in (0,1)$, we have
\begin{equation}\label{TLequiv01}
\|f\|_{F^{p,q}_\alpha} \approx    \|\mathcal{S}^{{\rm loc}, q}_{\alpha}f\|_{p}+\|f\|_{p},
 \end{equation}
while if $p,q\in [1,\infty]$ and $\alpha \in (0,1)$, then by~\cite[Theorem 9]{BPV-GAHA}
\begin{equation}\label{Besovequiv01}
 \|f\|_{B^{p,q}_{\alpha}} \approx \|f\|_{p} + \mathcal{A}^{p,q}_\alpha(f).
 \end{equation}

Let us stress that though the functionals $\mathcal{S}^{{\rm loc}, q}_{\alpha}$ and  $\mathcal{A}^{p,q}_\alpha$ are defined in~\cite{BPV-GAHA} in terms of a \emph{right} Haar measure while here in terms of $\lambda$, the two versions are equivalent as the modular function is bounded above and below away from $0$ on $B_1$.
 
For later purposes, we shall prove some properties of the finite
differences $\oD_{y}$ which will be of use. We first note that $\oD_y$
satisfies the following Leibniz rule: given two functions $f$ and $g$, 
\begin{equation}\label{leibnizD1}
\oD_y (fg)(x) = \oD_y g (x)f(x) + g(xy^{-1})\oD_{y}f(x),  \qquad x,y\in G.
\end{equation}
We observe moreover that, if $\phi$ is a function such that $\supp \phi \subseteq x B_r$ for some $x \in G$ and $r>0$, then for all  $y\in B_1$
\[
\supp \oD_{y} \phi \subseteq x B_{r+1}.
\]
\begin{lemma}\label{lemmader}
Suppose $p\in [1,\infty]$ and $|y|\leq 1$. Then 
\begin{itemize}
\item[\emph{(1)}] $\|\oD_{y} f \|_p \lesssim \| f\|_p$;
\item[\emph{(2)}] $ \|\oD_{y} f \|_p \lesssim |y|   \sum_{j=1}^{\kappa} \|X_j f\|_p$;
\item[\emph{(3)}] for all $k\in \N$ and $\psi \in C_c^{\infty}$ there exist $c=c(k)>0$ and $C(\psi)>0$ such that for all $t\in (0,1)$
\[
 \| \oD_y (\psi \Ls^{k}\e^{-t \Ls} f)\|_p \leq C(\psi) t^{-\frac{1}{2}-k}
 |y| \| \mathbf{1}_{\supp \psi}\,\e^{-ct \Ls} |f| \|_p,
\]
where $C(\psi)$ depends only on $\|\psi\|_{L^{\infty}_{1}}$ (see \eqref{Sobinfty}).
\end{itemize}
\end{lemma}

\begin{proof}
The proof of (1) is straightforward, since
\begin{align*}
\|\oD_{y} f \|_p \lesssim \| f(\cdot \, y^{-1}) \|_p + \|f\|_{p} \leq  (\delta^{1/p}(y) +1) \|f\|_p \lesssim \|f\|_p .
\end{align*}
We then prove (2), and argue as in the proof of~\cite[Theorem 3.1]{BPVBLMS}. Given $y\in B_{1}$, let $\gamma_y \colon [0,|y|] \to G$ be a horizontal subunit path such that $\gamma_y(0)=e$, $\gamma_y({|y|})=y^{-1}$, $|\gamma_y(s)| \leq |y|$ for every $s\in [0,|y|]$.

For every $x\in G$, by Taylor's formula applied to the function $s\mapsto f(x \gamma_y(s))$ and H\"older's inequality, one has
\begin{align*} 
|f(xy^{-1}) - f(x)|^p
&\leq \bigg(\int_0^{|y|} \sum_{{ j=1 }}^\kappa |X_jf (x\gamma_y(s))|\, \dd s\bigg)^p \nonumber \\
  & \leq |y|^{p-1}\int_0^{|y|}  \sum_{{  j=1}}^\kappa  |X_j f(x\gamma_y(s))|^p\, \dd s,
\end{align*}
so that
\begin{align*}
 \|\oD_{y} f \|_p^p 
 &\leq |y|^{p-1} \int_0^{|y|} \int_G   \sum_{{  j=1 }}^\kappa |X_j f(x\gamma_y(s))|^p\, \dd \lambda(x) \, \dd s \\
 & \lesssim  |y|^{p-1}  \sup_{s\in[0,|y|]}   \delta^{-p}(\gamma_y(s))  \int_0^{|y|}  \sum_{{  j=1 }}^\kappa\| X_j f \|_p^p \, \dd s \lesssim |y|^p  \sum_{{  j=1 }}^\kappa\| X_j f \|_p^p.
\end{align*}
To prove (3), observe that by (2)
\begin{align*}
\|\oD_y (\psi \e^{-t\Ls}\Ls^{k}f)\|_p
& \lesssim |y| \sum_{j=1}^{\kappa} \| X_j(\psi \Ls^{k}\e^{-t\Ls}f)\|_p,
\end{align*}
and for $t\in (0,1)$, by~\eqref{heatestimate}
\begin{align*}
 \| X_j(\psi \Ls^{k}\e^{-t\Ls}f)\|_p 
& \leq \| X_j \psi\cdot  \Ls^{k}\e^{-t\Ls}f\|_p + \| \psi X_j \Ls^{k}\e^{-t\Ls}f\|_p\\
& \lesssim  \|\mathbf{1}_{\supp \psi}\Ls^{k} \e^{-t\Ls}f\|_p + \| \mathbf{1}_{\supp \psi}  X_j \Ls^{k}\e^{-t\Ls}f\|_p\\
& \lesssim  t^{-k}\|\mathbf{1}_{\supp \psi} \e^{-c_{3}t\Ls}|f|\|_p + t^{-\frac12-k} \| \mathbf{1}_{\supp \psi} \e^{-c_{4}t\Ls}|f|\|_p\\
& \lesssim t^{-\frac12-k}\| \mathbf{1}_{\supp \psi}  \e^{-c t\Ls}|f|\|_p,
\end{align*}
for some $c_{3},c_{4},c>0$ by~\eqref{heatabovebelow}, and this completes the proof.
\end{proof}

\subsection{A covering lemma}
The following covering lemma will be used all throughout. It can be obtained as~\cite[Lemma 1]{Anker}, see also~\cite[Lemma 2.3]{B1}, with minor modifications. For the reader's convenience, we provide all the details. 
\begin{lemma}\label{covering}
There exists a countable family $\mathcal{U} = \{x_n\colon n\in \N\}\subset G$ such that
\begin{itemize}
\item[\emph{(1)}] $G=\bigcup_{n} x_n B_1$;
\item[\emph{(2)}] for all $m\in \N$ there exists $N_{m}\in \N$ such that each element of $G$ belongs to at most $N_{m}$ sets $xB_m$, $x\in  \mathcal{U}$;
\item[\emph{(3)}] for all $n\in \N$ and $m\in \N$ there are at most $N_{2m}$ elements $x\in \mathcal{U}$ such that $xB_m \cap x_{n}B_m$ is nonempty;
\item[\emph{(4)}]  for all $m\in \N$ there exist $N_{m}+1$ disjoint families of indices $I_{k}$, $k=1,\dots, N_{m}+1$ with the property that
\[
\N =\bigcup_{k=1}^{N_{m}+1} I_{k}, \qquad  \forall \, k=1,\dots, N_{m}+1, \; \;  d_C(x_{\ell},x_{h}) \geq m \quad \forall \ell,h \in I_{k}, \, \ell\neq h.
\] 

\end{itemize} 
\end{lemma}
\begin{proof}
By Zorn's lemma, there exists a countable maximal subset $\mathcal{U}$ of $G$ such that the sets $x B_{1/2}$,  $x\in \mathcal{U}$, are pairwise disjoint (recall that a connected Lie group with the topology of the Carnot--Carath\'eodory metric is second-countable, hence separable). Now, take any element $z\in G$. By maximality of $\mathcal{U}$, the set $z B_{1/2}$ meets at least one set $xB_{1/2}$, $x\in \mathcal{U}$. It follows that $z\in xB_{1/2} B_{1/2}^{-1} \subseteq  xB_1$ and (1) is proved.

\smallskip

Pick now $m\in \N$ and suppose that a set $ x_{0} B_m$ meets $N=N_{m}$ other sets $x_{1} B_m, \dots, x_{N} B_m$, with $x_{j}\in \mathcal{U}$. Then  $x_{0}B_mB_m^{-1} \ni x_{j}$, whence $x_{0}B_mB_m^{-1}B_{1/2}$ contains the sets $x_j B_{1/2}$, $j=0, \dots, N$, which are pairwise disjoint. It follows that
\begin{align*}
 \lambda (B_mB_m^{-1}B_{1/2})
&=  \lambda (x_{0}B_mB_m^{-1}B_{1/2}) \geq (1+N)  \lambda (B_{1/2})
\end{align*}
whence
\[
(1+N) \leq \frac{ \lambda (B_mB_m^{-1}B_{1/2})}{\lambda (B_{1/2})} \leq \frac{\lambda (B_{2m+1/2})}{\lambda (B_{1/2})},
\]
and (2) is proved.

\smallskip

To prove (3), observe that if $xB_{m} \cap x_{n}B_{m} \neq \emptyset$, then $d(x_{n},x) < 2m$, thus $x_{n} \in xB_{2m}$. By~(2), the number of such $x$'s is at most $N_{2m}$.

\smallskip

It remains to prove (4). Consider a maximal family $\mathcal{U}_{1}$ of points in $\mathcal{U}$ such that $x_{1} \in \mathcal{U}_{1}$ and $d(x_\ell, x_{h}) \geq m$ for all $x_{\ell},x_{h}\in \mathcal{U}_{1}$ with $\ell \neq h$. Then pick $x_{n_{2}}\in \mathcal{U} \setminus \mathcal{U}_{1}$, and consider a maximal family $\mathcal{U}_{2}$ of points in $\mathcal{U} \setminus \mathcal{U}_{1}$ such that $x_{n_2} \in \mathcal{U}_{2}$ and $d(x_\ell, x_{h}) \geq m$ for all $x_{\ell},x_{h}\in \mathcal{U}_{2}$ with $\ell \neq h$. Proceed recursively: at step $k$, consider $x_{n_{k}}\in \mathcal{U} \setminus \bigcup_{j=1}^{k-1} \mathcal{U}_{j}$ (if any) and consider a maximal family $\mathcal{U}_{k}$ of points in $\mathcal{U} \setminus \bigcup_{j=1}^{k-1}\mathcal{U}_{j}$ such that $x_{n_k} \in \mathcal{U}_{k}$ and $d(x_\ell, x_{h}) \geq m$ for all $x_{\ell},x_{h}\in \mathcal{U}_{k}$ with $\ell \neq h$.

Suppose by contradiction that one can proceed for more than $N_{m}+1$ steps. Then there exists an element $x_{n_{N_m+2}} \in \mathcal{U} \setminus \bigcup_{j=1}^{N_{m}+1} \mathcal{U}_{j}$; but by maximality of each of the $\mathcal{U}_{j}'s$, for all $j=1, \dots, N_{m}+1$ there is $\tilde x_{j}\in \mathcal{U}_{j}$ such that $d(x_{n_{N_m+2}} , \tilde x_j)<m $. Then
\[
x_{n_{N_{m}+2}} \in \tilde x_{j}B_m, \qquad \forall j=1,\dots, N_{m}+1,
\] 
and this contradicts (2). The required $I_{j}$'s are then the indices of the elements in $\mathcal{U}_{j}$.
\end{proof}
We shall not stress the dependence of $N$ on $m$ in the following, as this will not play any role. We shall refer to points~(2) and~(3) in Lemma~\ref{covering} as the \emph{bounded overlap property}.

\subsection{Pointwise multipliers}
We begin by setting some notation. First, we pick a smooth function $\eta$ on $G$ such that $0\leq \eta \leq 1$, $\eta=1$ on $B_1$ and $\supp \eta \subseteq B_2$. Such a function will be~\emph{fixed} all throughout. Then we consider the following family, to which $\eta$ belongs.

\begin{definition}\label{C-def}
We shall denote by $\sC$ the class of smooth cut-off functions 
\[
\sC:= \big\{ \xi \in C^\infty_c:\,  0\leq \xi \leq 1,\,  \xi=1 \text{ on } B_1\text{  and }   \supp \xi  \subseteq B_m \text{ for some }m\in\N\big\}.
\]
Given any $\xi \in\sC$, we set
\[
\tilde \xi_n = \xi(x^{-1}_n \cdot)  ,  \quad\text{ and } \quad  \xi_{n} = \frac{\tilde \xi_{n}}{\sum_{k} \tilde \xi_{k}},
\]
where $x_n \in \mathcal{U}$.
\end{definition}
By Lemma~\ref{covering}, for every $\xi \in \sC$ there exists $N\in \N$ such that
\begin{equation}\label{almostpartition}
1\leq \sum_{n\in \N}\xi_n(x) \leq N, \qquad x\in G,
\end{equation}
since for all $x\in G$ there are at most say $N$ nonzero terms in the sum above.  Hence, for all $\xi \in \sC$ 
\[
\sum_{n} \xi_{n}=1, \qquad \supp \xi_{n} = \supp \tilde \xi_{n} = x_{n}\supp \xi \subseteq x_{n}B_{m},
\]
where $m \in \N$ is such that $\supp \xi \subseteq B_m$, and still for all $x\in G$ there are at most $N$ nonzero terms in the sum for some $N\in \N$. Though it is not true that $\xi_{n}$ is a (left) translate of $\xi$, it is still true that for $p\in [1,\infty]$
\begin{equation}\label{unifbound}
\sup_{n} \sup_{|J|=m} \| X_{J}\xi_{n}\|_{p} <\infty.
\end{equation}
Such an estimate is a consequence of~\eqref{almostpartition}, the fact that $\tilde \xi_{n}$ is a left translate of $\xi$ and the left invariance of the norm and of the vector fields $X_j$. 

In particular, all the above holds for $\eta$.

\begin{definition}
Suppose $p,q\in [1,\infty]$ and $\alpha>0$. We say that a function $f$ is uniformly locally in $X^{p,q}_\alpha$, and we write $f\in X^{p,q}_{\alpha, \unif}$, if  
\[
\| f \|_{X^{p,q}_\alpha, \unif} = \sup_{n\in \N}\| f \eta_n \|_{X^{p,q}_\alpha}<\infty.
\]
We denote by $MX^{p,q}_\alpha$  the space of multipliers of $X^{p,q}_\alpha $, namely the space of functions $f$ such that $\|fg\|_{X^{p,q}_\alpha} \leq C(f) \|g\|_{X^{p,q}_\alpha} $ for all $g \in X^{p,q}_\alpha$, endowed with the norm $\|f\|_{MX^{p,q}_\alpha} $ of the infimum of all such $C(f)$.
\end{definition}
In the following lemma we prove few basic facts which will be of use all throughout. In particular, we show that for the range of indices which we shall be interested in the definition of $X^{p,q}_{\alpha, \unif}$ is independent of the choice of $\eta$. In other words, if one replaces $\eta$ with any other $\xi \in \sC$, then the two norms are equivalent.

\begin{lemma}\label{lemmabasic}
Suppose $\xi,\phi \in \sC$ and let $f$ be a function. Then the following holds. 
\begin{itemize}
\item[\emph{(1)}]
  For all $p\in [1,\infty)$ 
\begin{equation}\label{maindec}
\|f\|_{p}^{p} \approx  \sum_{n\in\N}\|f\xi_{n}\|_{p}^{p}, \qquad \|f\|_{\infty} \approx \sup_{n\in\N}\|f\xi_{n}\|_{\infty}.
\end{equation}
\item[\emph{(2)}] For $p,q\in [1,\infty]$ and $\alpha >d/p$ or $p\in
  [1,\infty]$, $q=1$ and $\alpha =d/p$ if $X=B$, and  $p,q\in
  (1,\infty)$ and $\alpha >d/p$ if $X=F$, 
  \begin{equation}\label{equivsenzasup}
\begin{split}
 &\|f\tilde \xi_{n}\|_{X^{p,q}_\alpha} \approx \| f \xi_n \|_{X^{p,q}_\alpha} \quad \forall \, n\in\N,  \\
 &\sup_{n\in \N}\| f \xi_n \|_{X^{p,q}_\alpha} \approx \sup_{n\in \N}\| f \phi_n \|_{X^{p,q}_\alpha} .
\end{split}
\end{equation}
\item[\emph{(3)}] If $J$ is a multi-index and $p,q$ are as in {\rm (2)}, then
\begin{align}\label{equivaderiv}
\sup_{n\in\N} \|\xi_{n}X_{J}f\|_{X^{p,q}_\alpha} \lesssim \sup_{n\in \N}\|X_{J} (f \xi_n) \|_{X^{p,q}_\alpha}.
 \end{align}
\end{itemize}
\end{lemma}
\begin{proof}
To prove (1), consider first the case $p<\infty$ and observe that
\[
\|f\|_{p}^{p} = \int_{G} |f|^{p}\, \dd \lambda =  \int_{G} \Big(\sum_{n}|f| \xi_{n}\Big)^{p}\, \dd \lambda.
\]
For all $x\in G$, by Lemma~\ref{covering} there are at most $N$ functions $\xi_{n_1^{x}}, \dots, \xi_{n_{N}^{x}}$ such that $\xi_{n_{j}^{x}}(x)\neq 0$, with $N$ independent of $x$. Thus,
\begin{equation}\label{intp}
\int_{G} \Big(\sum_{n}|f| \xi_{n}\Big)^{p}\, \dd \lambda   \approx  \int_{G} \sum_{n}|f \xi_{n}|^{p}\, \dd \lambda  =   \sum_{n}\|f\xi_{n}\|_{p}^{p},
\end{equation}
where the constants depend only on $p$ and $N$. The case $p=\infty$ is similar: the inequality 
\[
\|f\xi_{n}\|_{\infty} \leq \|\xi_{n}\|_{\infty}\|f\|_{\infty} \lesssim \|f\|_{\infty}
\]
follows by~\eqref{unifbound}; moreover, for $x\in G$
\[
|f|(x) = \sum_{n}|f|(x)\xi_{n}(x)  =  \sum_{j=1}^{N}|f|(x)\xi_{n_{j}^{x}}(x)  \leq N \sup_{n} \sup_{x} |f(x)\xi_{n}(x)|,
\]
so that also the other inequality follows.

The equivalences stated in (2) and (3) are consequences of the algebra property, cf.~\eqref{algebra} (wherefrom the restriction on the indices). Let $\psi\in\sC$ be  such that $\psi= 1$ on $\supp \xi$. Observe that $\xi=\psi\xi$, whence $\tilde \xi_{k} = \tilde \xi_{k}\tilde \psi_{k}$ and  $\xi_{k} = \xi_{k}\tilde \psi_{k}$ for all $k\in\N$. 

For $n\in \N$, by~\eqref{algebra}
\begin{align*}
\| f\tilde \xi_{n}\|_{X^{p,q}_{\alpha}} 
&= \Big\|\Big(\sum_{m} \tilde \xi_{m}\Big) f \xi_{n} \Big\|_{X^{p,q}_{\alpha}} \\
& =  \Big\|\Big(\sum_{m} \tilde \xi_{m} \Big)\tilde \psi_{n} f \xi_{n}\Big\|_{X^{p,q}_{\alpha}}   \lesssim  \Big\|\Big(\sum_{m} \tilde \xi_{m} \Big)\tilde \psi_{n}\Big\|_{X^{p,q}_{\alpha}}  \| f \xi_{n}\|_{X^{p,q}_{\alpha}}.
\end{align*}
By left-invariance, the bounded overlap property,  and the algebra  property we now have
\begin{align*}
\Big\|\Big(\sum_{m} \tilde \xi_{m} \Big)\tilde \psi_{n}\Big\|_{X^{p,q}_{\alpha}}
 &  =   \Big\|\Big(\sum_{j=1}^N \tilde \xi_{m_j} \Big)\psi  \Big\|_{X^{p,q}_{\alpha}} \\
&\leq \Big\|\Big(\sum_{j=1}^N \tilde \xi_{m_j} \Big)\Big\|_{X^{p,q}_{\alpha}}    \| \psi \|_{X^{p,q}_{\alpha}} \leq N   \| \xi \|_{X^{p,q}_{\alpha}}   \| \psi \|_{X^{p,q}_{\alpha}},
\end{align*}
the last quantity being finite since it is the norm of a smooth and compactly supported function. 
We conclude that
\[
\| f\tilde \xi_{n}\|_{X^{p,q}_{\alpha}}  \lesssim \| f \xi_{n}\|_{X^{p,q}_{\alpha}}.
\]
To prove the converse inequality in~\eqref{equivsenzasup} we first write
\begin{align*}
\| f  \xi_{n}\|_{X^{p,q}_{\alpha}} 
&= \Big\|\Big(\sum_{m} \tilde \xi_{m}\Big)^{-1} f \tilde \xi_{n} \Big\|_{X^{p,q}_{\alpha}} \\
& =  \Big\|\Big(\sum_{m} \tilde \xi_{m} \Big)^{-1}\tilde \psi_{n} f \tilde \xi_{n}\Big\|_{X^{p,q}_{\alpha}}   \lesssim  \Big\|\Big(\sum_{m} \tilde \xi_{m} \Big)^{-1}\tilde \psi_{n}\Big\|_{X^{p,q}_{\alpha}}  \| f \tilde \xi_{n}\|_{X^{p,q}_{\alpha}},
\end{align*}
again by~\eqref{algebra}. For $\alpha>d/p$, by~\cite[Theorems 5.1, 5.2, 5.3]{BPV-MA}, we can find a positive integer $k$ such that $L^p_{k} \hookrightarrow X^{p,q}_{\alpha}$, so that by~\eqref{Sobinteger}
\[
\begin{aligned}
\Big\|\Big(\sum_{m} \tilde \xi_{m} \Big)^{-1}\tilde \psi_{n}\Big\|_{X^{p,q}_{\alpha}}&\lesssim \Big\|\Big(\sum_{m} \tilde \xi_{m} \Big)^{-1}\tilde \psi_{n}\Big\|_{L^{p}_{k}}\\
&\lesssim \sum_{|I|\leq k} \Big\|X_I \Big(\Big(\sum_{m} \tilde \xi_{m} \Big)^{-1}\tilde \psi_{n} \Big) \Big\|_{p}\\
&\lesssim  \sum_{|I|+|J|\leq k} \Big\|X_I \Big(\Big(\sum_{m} \tilde \xi_{m} \Big)^{-1}\Big)X_J\big(\tilde \psi_{n} \big) \Big\|_{p} \lesssim 1,
\end{aligned}
\]
by the left-invariance of the vector fields. This concludes the proof of the first equivalence in~\eqref{equivsenzasup}. 

We now prove the second equivalence; by symmetry, it is enough to prove one of the two inequalities. For $n\in\N$, let $m_{1}^{n}, \dots, m_{M}^{n}$ be the indices such that $x_{m_{j}^{n}}\supp \xi$ intersects $x_{n}\supp \phi$. The number $M$ depends only on $\xi$ and $\phi$, but not on $n$, by the bounded overlap property. Then
\[
f \tilde \xi_n = \sum_{j=1}^{M} f \tilde \xi_{n} \phi_{m_{j}^{n}}, 
\]
whence, by using~\eqref{equivsenzasup} and the left-invariance of the norms,
\begin{align*}
\|f \xi_n\|_{X^{p,q}_{\alpha}}  \lesssim \| f\tilde \xi_{n}\|_{X^{p,q}_{\alpha}}
& = \Big\| \tilde \xi_{n} \sum_{j=1}^{M} f \phi_{m_{j}^{n}}\Big\|_{X^{p,q}_{\alpha}} \\
& \lesssim \|\tilde \xi_{n}\|_{X^{p,q}_{\alpha}}  \Big\| \sum_{j=1}^{M} f \phi_{m_{j}^{n}}\Big\|_{X^{p,q}_{\alpha}} \\
& \leq  \| \xi\|_{X^{p,q}_{\alpha}}  M \sup _{j=1, \dots, M } \| f \phi_{m_{j}^{n}}\|_{X^{p,q}_{\alpha}} \lesssim \sup _{m} \| f \phi_{m}\|_{X^{p,q}_{\alpha}},
\end{align*}
which completes the proof of (2).

To prove (3), observe that for $n\in \N$ there are $k_{j}^{n}$, $j=1, \dots, N$, such that
\begin{align*}
\|\xi_{n}X_{J}f\|_{X^{p,q}_{\alpha}}  = \Big\| \xi_{n}\sum_{j=1}^{N} X_{J}(f\xi_{k_{j}^{n}}) \Big\|_{X^{p,q}_{\alpha}}    \lesssim \sum_{j=1}^{N}\|  \xi_{n} X_{J}(f\xi_{k_{j}^{n}}) \|_{X^{p,q}_{\alpha}}
\end{align*}
and by~\eqref{algebra}
\[
\|  \xi_{n} X_{J}(f\xi_{k_{j}^{n}}) \|_{X^{p,q}_{\alpha}} \lesssim  \|  \tilde \xi_{n}\|_{X^{p,q}_{\alpha}} \|X_{J} (f\xi_{k_{j}^{n}}) \|_{X^{p,q}_{\alpha}} \lesssim \sup_{k} \|X_{J} (f\xi_{k}) \|_{X^{p,q}_{\alpha}} ,
\]
where we used that $ \| \xi_{n}\|_{X^{p,q}_{\alpha}} \lesssim  \| \tilde  \xi_{n}\|_{X^{p,q}_{\alpha}} = \|\xi\|_{X^{p,q}_{\alpha}}$. This completes the proof.
\end{proof}

We are now ready to show that a multiplier of $X^{p,q}_{\alpha}$, for $\alpha >d/p$, belongs to $X^{p,q}_{\alpha,\unif}  $.
\begin{proposition}\label{proponedir}
Suppose $\alpha>d/p$ and $p,q\in [1,\infty]$ if $X=B$ or $p,q\in (1,\infty)$ if $X=F$. Then $MX^{p,q}_\alpha \hookrightarrow X^{p,q}_{\alpha,\unif}  $.
\end{proposition}
\begin{proof}
Pick $f\in MX^{p,q}_\alpha $ and observe that by~\eqref{equivsenzasup}
\begin{align*}
\sup_{n}\| f \eta_n\|_{X^{p,q}_\alpha } 
& \leq \sup_{n}\|f\|_{MX^{p,q}_\alpha} \| \eta_n\|_{X^{p,q}_\alpha}  \\
& \lesssim \|f\|_{MX^{p,q}_\alpha}  \sup_{n} \| \tilde \eta_n\|_{X^{p,q}_\alpha} =  \|f\|_{MX^{p,q}_\alpha} \| \eta \|_{X^{p,q}_\alpha},
\end{align*}
and the statement follows.
\end{proof}

\section{Multipliers of Triebel--Lizorkin spaces}\label{Sec:TL}
In this section, inspired by~\cite{Str-67}, we shall prove the following.
\begin{theorem}\label{teoTL}
Suppose $p,q\in (1,\infty)$ and $\alpha>d/p$. Then $MF^{p,q}_\alpha = F^{p,q}_{\alpha, \unif}$ with equivalences of norms.
\end{theorem}
We begin with the following proposition, which in particular provides an equivalent characterization of the Triebel--Lizorkin norm of a function by means of the localizing functions in $\sC$. 

\begin{proposition}\label{propsecondirTL}
Suppose $p,q\in (1,\infty)$ and $\alpha>0$, and let $\{\varphi_{(n)}\}$ be a sequence of smooth functions such that $\supp(\varphi_{(n)})\subseteq x_nB_m$ for some $m\in \N$, where $\{x_n \colon n\in \N\}=\cU$ is as in Lemma~\ref{covering}, and with all derivatives of order $\leq \alpha +1$ along $\bf X$ uniformly bounded. Then, for every $f\in F^{p,q}_\alpha$, 
\begin{equation}\label{sumbricks-from-above}
 \Big(\sum_{n\in\N} \| f
 \varphi_{(n)}\|^p_{F^{p,q}_\alpha}\Big)^{1/p} \lesssim \|f\|_{F^{p,q}_\alpha} .
\end{equation}
If $\xi\in\sC$, then
\begin{equation}\label{sumbricks}
\|f\|_{F^{p,q}_\alpha} \approx\Big(\sum_{n\in\N} \| f \xi_n\|^p_{F^{p,q}_\alpha}\Big)^{1/p}.
\end{equation}
\end{proposition}
\begin{proof}
Assume first that $\alpha \in (0,1)$ and pick $f\in F^{p,q}_\alpha$. By~\eqref{leibnizD1}
\begin{align*}
\mathcal{S}^{{\rm loc}, q}_{\alpha}(\varphi_{(n)} f)(x) 
  & \leq \bigg(\int_0^1\bigg[ \frac{1}{u^{\alpha} V(u)}\int_{|y|<u}|\oD_{y}f(x)| |\varphi_{(n)}(x)| \, \dd\lambda(y) \bigg]^q\,
  \frac{d u}{u}\bigg)^{1/q}  \\
  &\quad + \bigg(\int_0^1\bigg[ \frac{1}{u^{\alpha}
      V(u)}\int_{|y|<u}|f(xy^{-1})\oD_{y}\varphi_{(n)}(x)| \, \dd\lambda(y) \bigg]^q\,
  \frac{d u}{u}\bigg)^{1/q}\\
  & = I_n(x) + J_n(x).
\end{align*}
On the one hand,
\begin{align*}
\sum_{n} \|I_{n}\|_{p}^{p} = \int_{G}\sum_{n}   |I_n|^p \, \dd \lambda \lesssim \bigg(\sup_{x\in G} \sum_{n}   |\varphi_{(n)}(x)| \bigg)^{p} \| \mathcal{S}^{{\rm loc}, q}_{\alpha}( f)\|^p_{p}  \lesssim \| \mathcal{S}^{{\rm loc}, q}_{\alpha}( f)\|^p_{p}.
\end{align*}
On the other hand, since $\supp \oD_{y}\varphi_{(n)} \subseteq x_{n}B_{m+1}$ when $|y|\leq 1$, one has $J_n = J_{n }\mathbf{1}_{x_{n}B_{m+1}}$. By Lemma~\ref{lemmader}~(2) applied to $\varphi_{(n)}$ with $p=\infty$,
\begin{align*}
J_{n}(x) &\lesssim \mathbf{1}_{x_{n}B_{m+1}}(x) \bigg(\int_0^1 \bigg[ \frac{1}{u^{\alpha}  V(u)}\int_{|y|<u}|f(xy^{-1})|y| \, \dd\lambda(y) \bigg]^q\, \frac{\dd u}{u} \bigg)^{1/q}\\
& \lesssim \bigg(\int_0^1 \bigg[ \frac{u}{u^{\alpha}}|f\mathbf{1}_{x_{n}B_{m+2}}|*g_{u}(x) \bigg]^q\, \frac{\dd u}{u} \bigg)^{1/q}\\
& \lesssim \sup_{u\in (0,1)} |f \mathbf{1}_{x_{n}B_{m+2}}|*g_{u}(x),
\end{align*}
where $g_{u}(y) = \frac{1}{V(u)}\mathbf{1}_{B_u}(y)$. Then, by Young's inequality (cf.~\cite[(2.2)]{BPV-NA}),
\[
\|J_{n}\|_{p}^{p} \lesssim \| f\mathbf{1}_{x_{n}B_{m+2}}\|_{p}^{p}  \sup_{u\in (0,1)} \|g_{u}\|_{1}^{p} \leq \| f\mathbf{1}_{x_{n}B_{m+2}}\|_{p}^{p} .
\]
Thus,
\begin{align*}
\sum_{n}  \|J_{n}\|_{p}^{p}  \lesssim  \|f\|_{p}^p.
\end{align*}
Therefore,  using~\eqref{TLequiv01}, the bounded overlap property and the uniform bound on $\|\varphi_{(n)}\|_\infty$, 
\begin{align*}
\Big(\sum_{n} \| f \varphi_{(n)}\|^p_{F^{p,q}_\alpha}\Big)^{1/p}
& \lesssim \Big(\sum_{n} \| f \varphi_{(n)}\|^p_p\Big)^{1/p} + \Big(\sum_{n} \| \mathcal{S}^{{\rm loc}, q}_{\alpha}(\varphi_{(n)} f)\|^p_p\Big)^{1/p}\notag\\
& \lesssim \|f\|_p + \| \mathcal{S}^{{\rm loc}, q}_{\alpha}(f)\|_p \lesssim \|f\|_{F^{p,q}_\alpha}.
\end{align*}
This shows the inequality~\eqref{sumbricks-from-above} in the case
$\alpha\in(0,1)$.
If $\alpha = k+ \alpha'$ with $k\in \N$ and $\alpha'\in (0,1)$, then by~\cite[Theorem 4.5]{BPV-MA}
\begin{equation}\label{recursiveTL}
\|f\|_{F^{p,q}_\alpha} \approx \sum_{|I|\leq k} \| X_I f\|_{F^{p,q}_{\alpha'}}, \qquad \|f\varphi_{(n)}\|_{F^{p,q}_\alpha} \approx \sum_{|I|\leq k} \| X_I (f\varphi_{(n)})\|_{F^{p,q}_{\alpha'}}.
\end{equation}
Since
\[
 X_I (f \varphi_{(n)})  =  \sum_{|J|\leq |I|, |L| \leq |I|-|J|}  c_{J,L}(X_J f) (X_L \varphi_{(n)}),
\]
 arguing as in the case $\alpha\in(0,1)$,  with $X_J f$ in place of  $f$ and with $X_L \varphi_{(n)}$ in place of $\varphi_{(n)}$, we obtain  
\[
\Big( \sum_{n} \| (X_J f) (X_L\varphi_{(n)}) \|_{F^{p,q}_{\alpha'}}^p\Big)^{1/p} \lesssim  \|X_J f\|_{F^{p,q}_{\alpha'}} \lesssim \|f\|_{F^{p,q}_{\alpha}} ,
\]
whence
\[
\Big(\sum_{n} \| f\varphi_{(n)}\|_{F^{p,q}_{\alpha}}^p \Big)^{1/p} \lesssim \|f\|_{F^{p,q}_\alpha}.
\]
Thus,  inequality~\eqref{sumbricks-from-above} follows for all $\alpha>0$ which are not integers.
The integer case follows from interpolation. Indeed, if  $k\in\N$, then by~\cite[\S 5.6]{BerghLofstrom} and~\cite[Theorem 6.1]{BPV-MA}, 
\[
(\ell^{p}(F^{p,q}_{k/2}), \ell^{p}(F^{p,q}_{3k/2}))_{[1/2]} = \ell^{p}((F^{p,q}_{k/2}, F^{p,q}_{3k/2})_{[1/2]}) = \ell^{p} (F^{p,q}_k),
\]
but also, by what shown above,
\[
(\ell^{p}(F^{p,q}_{k/2}), \ell^{p}(F^{p,q}_{3k/2}))_{[1/2]} = (F^{p,q}_{k/2}, F^{p,q}_{3k/2})_{[1/2]} =F^{p,q}_k,
\]
all with equivalences of norms.  This proves~\eqref{sumbricks-from-above}.

 In order to prove~\eqref{sumbricks} we only need to prove the reverse inequality, assuming that $\xi\in\sC$ and $\xi_n$ is as in Definition~\ref{C-def}.  We have  that
\begin{align*}
\mathcal{S}^{{\rm loc}, q}_{\alpha}( f) = \mathcal{S}^{{\rm loc}, q}_{\alpha}\Big( \sum_{n}f\xi_{n}\Big)
& \leq \sum_{n} \mathcal{S}^{{\rm loc}, q}_{\alpha}( f\xi_{n}).
    \end{align*}
We observe that, if $\{g_{(n)}\}$ is a sequence of nonnegative functions such that $\supp g_{(n)}\subseteq x_nB_m$ for some $m\in\N$, with $\{x_n\}$ as in Lemma~\ref{covering}, then
\begin{equation}  \label{ci-serve-no?}  
\Big\|\sum_{n} g_{(n)}\Big\|_p \approx \bigg( \sum_{n} \|g_{(n)}\|_p^p \bigg)^{1/p}
\end{equation}
Thus, if the right hand side of~\eqref{sumbricks} is finite, since $\supp \mathcal{S}^{{\rm loc}, q}_{\alpha}( f\xi_n) \subseteq x_{n}B_{m+2}$,  for $\alpha \in (0,1)$ we have
\begin{align*}
\|f\|_p +\|\mathcal{S}^{{\rm loc}, q}_{\alpha}( f)\|_p 
& \lesssim  \bigg(\sum_{n}\|f\xi_{n}\|_{p}^{p}\bigg)^{1/p} + \Big(\sum_{n} \|\mathcal{S}^{{\rm loc}, q}_{\alpha}( f\xi_n)\|^p_p\Big)^{1/p},
\end{align*}
which gives the desired conclusion when $\alpha \in (0,1)$. To conclude, suppose now $\alpha>0$ is noninteger and $\alpha = k+ \alpha'$ with $\alpha'\in (0,1)$ and $k\in \N$. Then, again by the norm equivalence~\eqref{recursiveTL}, we observe that
\[
\mathcal{S}^{{\rm loc}, q}_{\alpha '}(X_{I} f) \leq \sum_{n} \mathcal{S}^{{\rm loc}, q}_{\alpha'}(X_{I} (f\xi_{n}))
\]
and that $\supp \mathcal{S}^{{\rm loc}, q}_{\alpha}( X_{I}(f\xi_n))
\subseteq x_{n}B_{m+2}$, so that for $|I|\leq k$, arguing as in~\eqref{ci-serve-no?}, 
\begin{align*}
\| X_I f\| _{F^{p,q}_{\alpha'}}
 & \approx \| X_I f\|_{p} +  \| \mathcal{S}^{{\rm loc}, q}_{\alpha '}(X_{I} f) \|_{p}\\
 & \lesssim  \Big(\sum_{n} \|X_{I}(f\xi_{n})\|_p^{p}\Big)^{1/p} +\Big(\sum_{n} \|\mathcal{S}^{{\rm loc}, q}_{\alpha'}X_{I}( f\xi_n)\|^p_p\Big)^{1/p} \\
 &\lesssim  \Big(\sum_{n}  \|X_{I}(f\xi_{n})\|_{F^{p,q}_{\alpha'}}^{p}\Big)^{1/p} \lesssim \Big(\sum_{n}   \|f\xi_{n}\|_{F^{p,q}_{\alpha}}^{p}\Big)^{1/p}.
 \end{align*}
 Thus, the statement follows for all $\alpha>0$ which are not
 integers. The integer case follows from interpolation as before.
\end{proof}

\begin{corollary}\label{corsecondirTL}
Suppose that $p,q\in (1,\infty)$, $\alpha>0$, and $\xi \in\sC$. Then
\begin{equation}\label{supfetan}
\|f \|_{MF^{p,q}_\alpha} \approx \sup_{n\in \N}\|f\xi_n\|_{MF^{p,q}_\alpha}.
\end{equation}
\end{corollary}
\begin{proof}
On the one hand, by Proposition~\ref{propsecondirTL}, for $g\in F^{p,q}_{\alpha}$ one has 
\[
\|\xi_{n}g\|_{F^{p,q}_{\alpha}} \lesssim \|g\|_{F^{p,q}_{\alpha}}, \qquad n\in\N,
\]
whence $\|\xi_n\|_{MF^{p,q}_\alpha} \lesssim 1$ uniformly for $n\in \N$. Then,
\[
\|f\xi_n\|_{MF^{p,q}_\alpha} \leq \|f \|_{MF^{p,q}_\alpha}\|\xi_n\|_{MF^{p,q}_\alpha} \lesssim \|f\|_{MF^{p,q}_\alpha},
\]
which is a bound independent of $n$, from which the inequality
$\gtrsim$.

Conversely, let $\psi\in\sC$ be such that $\psi=1$ on $\supp\xi$, so that $\xi_n=\xi_n\widetilde\psi_n$. Now, if the right hand side of~\eqref{supfetan} is finite,  and $g\in F^{p,q}_\alpha$,   by Proposition~\ref{propsecondirTL} we have
\begin{align*}
\| fg\|_{F^{p,q}_\alpha} 
&\lesssim  \Big( \sum_{n} \| f\xi_n g\|_{F^{p,q}_{\alpha}}^p\Big)^{1/p}\\
&= \Big( \sum_{n} \| f\xi_n\widetilde\psi_n g\|_{F^{p,q}_{\alpha}}^p\Big)^{1/p}\\
  & \leq \sup_{n}\|f\xi_n\|_{MF^{p,q}_\alpha}  \Big( \sum_{n} \| \widetilde\psi_n g\|_{F^{p,q}_{\alpha}}^p\Big)^{1/p}  \lesssim  \sup_{n}\|f\xi_n\|_{MF^{p,q}_\alpha}   \|g\|_{F^{p,q}_\alpha}
\end{align*}
whence the inequality $\lesssim $ in~\eqref{supfetan}.
\end{proof}

\begin{proof}[Proof of Theorem~\ref{teoTL}]
If $f\in MF^{p,q}_\alpha$, then $f$ is uniformly locally in
$F^{p,q}_\alpha$ by Proposition~\ref{proponedir}. Viceversa, assume that $f$ is uniformly locally in $F^{p,q}_\alpha$. Let $\psi\in\sC$ be such that $\psi=1$ on $\supp\eta$, so that $\eta_n=\eta_n\widetilde\psi_n$. Then for $g\in F^{p,q}_\alpha$, by the algebra property of $F^{p,q}_\alpha$ (recall~\eqref{algebra}) and Proposition~\ref{propsecondirTL} 
\begin{align*}
\|fg\|_{F^{p,q}_\alpha} &\lesssim \bigg(\sum_{n} \|f\eta_n\widetilde\psi_n g\|_{F^{p,q}_\alpha}^p \bigg)^{1/p}\\
 & \lesssim \sup_n \|f\eta_n\|_{F^{p,q}_\alpha}
 \bigg( \sum_{n} \| \widetilde\psi_n g\|_{F^{p,q}_\alpha}^p \bigg)^{1/p}  \lesssim \bigg(\sum_{n} \| \widetilde\psi_n  g\|_{F^{p,q}_\alpha}^p \bigg)^{1/p}  \lesssim \| g\|_{F^{p,q}_\alpha} ,
\end{align*} 
and the theorem is proved.
\end{proof}

\section{Multipliers for Besov spaces}\label{Sec:Besov}
In this section we deal with the pointwise multipliers for the spaces $B^{p,q}_{\alpha}$. On the one hand, as already explained, the Besov case is intrinsically different (and more difficult) than the Triebel--Lizorkin case, and the pointwise multipliers differ depending on whether $q\geq p$ or $q<p$. On the other hand, because of the lack of a suitable notion of high-order differences in the full generality of a Lie group as in the previous sections, we are unable to get a characterization of $MB^{p,q}_\alpha$ for all indices $p,q\in [1,\infty]$ and $\alpha>d/p$ (which will instead be obtained in Section~\ref{Sec:stratified} below, when $G$ is stratified). We shall slightly restrict the ranges of $p,q$ and $\alpha$ involved, and get the following.
\begin{theorem}\label{teoB}
Suppose $p,q\in [1,\infty]$, $d<p\leq q$ and $\alpha\in (d/p, 1)+\N$. Then $MB^{p,q}_\alpha = B^{p,q}_{\alpha,\unif}$, with equivalence of norms.
\end{theorem}
The case when $p=q=1$ and $\alpha=d$ is somewhat special, as $B^{1,1}_{d} = F^{1,1}_{d}$ is also a Triebel--Lizorkin space and it is also an algebra (cf.~\cite[Corollary 7.2]{BPV-MA}). Thus, it can be treated as in the previous section. We have the following.
\begin{theorem}\label{teoB11}
$MB^{1,1}_{d} =  B^{1,1}_{d,\unif}$ with equivalence of norms.
\end{theorem}
The theorems above cover the case $q\geq p$. Before we describe the case $q<p$, we give the following.
\begin{definition}
For $\alpha>0$ and  $ p,q\in [1,\infty]$, we shall denote by $M_{\alpha}^{p,q}$ the space of all $f \in L^{1}_{\mathrm{loc}}$ such that
\[
\|f\|_{M_{\alpha}^{p,q}} =  \sup_{\| \gamma_{n}\|_{\ell^{p}}\leq 1} \Big\| \sum_{n}\gamma_{n}\eta_{n}f\Big\|_{B^{p,q}_{\alpha}}
\]
is finite, endowed with the above norm.
\end{definition}

Then we have the following.
\begin{theorem}\label{q<p}
Suppose $1\leq q < p < \infty$, $p>d$ and $\alpha \in (d/p,1)$. Then  $M B_{\alpha}^{p,q} = M_{\alpha}^{p,q}$ with equivalence of norms.
\end{theorem}

The reason of the restriction on the $\alpha$'s will become clear soon, and is due to the fact that we use first-order differences only. The approach followed for Triebel--Lizorkin spaces does work in this case, unless $p=q$ (which amounts to $B^{p,p}_{\alpha} = F^{p,p}_{\alpha}$), since the analogue of Proposition~\ref{propsecondirTL} fails for Besov spaces.

The remaining part of the Section is devoted first to some technical results, then to the proofs of Theorems \ref{teoB}, \ref{teoB11}, and \ref{q<p}. 

\subsection{Some equivalences of norms}
We shall need a characterization in the same spirit as~\eqref{Besovequiv01} which involves the modulus of smoothness $\omega_{1}$, defined as
\[
 \omega_1 (f,t,p) = \sup_{|y|<t} \|\oD_{y} f \|_p \qquad t>0.
\]
 We begin with a few lemmas.

\begin{lemma}\label{equivfirst}
Suppose $p,q\in [1,\infty]$ and $\alpha \in (0,1)$. Then
\begin{align*}
\|f\|_p + \bigg(\int_0^1 (t^{-\alpha}  \sup_{|y|<t} \|\oD_{y} f \|_p )^q  \frac{\dd t}{t}\bigg)^{1/q}
& \approx  \|f  \|_p + \bigg( \sum_{k\in \N}  (2^{k\alpha}  \omega_1 (f, 2^{-k},p))^q\bigg)^{1/q}.
\end{align*}
\end{lemma}
\begin{proof}
It is just a standard discretization and reconstruction of the integral.
\end{proof}

\begin{lemma}\label{chanineq}
Suppose $p,q\in [1,\infty]$ and $\alpha \in (0,1)$. Then
\[
 \|f\|_{p} + \As^{p,q}_\alpha(f) \approx \|f\|_{p} + \bigg( \sum_{k\in \N}  (2^{k\alpha}  \omega_1 (f, 2^{-k},p))^q\bigg)^{1/q} \approx \|f\|_{B^{p,q}_{\alpha}}.
\]
\end{lemma}

\begin{proof}
By~\eqref{Besovequiv01}, it will be enough to prove the chain of inequalities
\[
 \|f\|_{p} + \As^{p,q}_\alpha(f) \lesssim \|f\|_{p} + \bigg( \sum_{k}  (2^{k\alpha}  \omega_1 (f, 2^{-k},p))^q\bigg)^{1/q} \lesssim \|f\|_{B^{p,q}_{\alpha}}.
\]
The first inequality can be easily seen by decomposing the ball $|y|\leq 1$ into annuli $2^{-k} \leq |y|\leq 2^{-k+1}$, $k\in \N$. To show the second inequality, by Lemma~\ref{equivfirst} it is enough to prove that for $m\geq 1$
\[
\bigg(\int_0^1 (t^{-\alpha}  \sup_{|y|<t} \|\oD_{y} f \|_p )^q  \frac{\dd t}{t}\bigg)^{1/q}  \lesssim \|f\|_p + \bigg(\int_0^1 \!(t^{-\alpha/2}  \| (t\Ls )^{m}\e^{-t\Ls} f \|_p )^q \, \frac{\dd t}{t}\bigg)^{1/q}  .
\]
Recall that, see e.g.~\cite[(4.1)]{BPV-MA},
\[
f= \frac{1}{(m-1)!} \int_0^{1} (s\Ls)^{m} \e^{-s\Ls} f \, \frac{\dd s}{s} +  \sum_{\ell=0}^{m-1} \frac{1}{\ell!}  \Ls^\ell \e^{-\Ls} f.
\]
The second term is easily dealt with, as by Lemma~\ref{lemmader}~(2) and the $L^{p}$-boundedness of the heat semigroup
\[
\|\oD_{y} (\Ls^\ell \e^{-\Ls} f) \|_p \lesssim |y| \sum_{j=1}^{\ell} \|X_{j} \Ls^\ell \e^{-\Ls} f\|_{p} \lesssim |y|\|f\|_{p}.
\]
As for the first term, we note that by~\eqref{heatestimate} and again the $L^{p}$-boundedness of the heat semigroup
\[
\|X_{j} (s\Ls)^{m} \e^{-s\Ls} f\|_p = \|X_{j}  \e^{-s\Ls}(s\Ls)^{m} f\|_p  \lesssim s^{-\frac12} \| \e^{-\frac{s}{2}\Ls}(s\Ls)^{m} f\|_p ,
\]
whence by Lemma~\ref{lemmader}~(1) and~(2)
\begin{align*}
\sup_{|y|<t} \bigg\|\oD_{y} \int_0^{1} (s\Ls)^{m} \e^{-s\Ls} f \, \frac{\dd s}{s}   \bigg\|_p 
&\leq \int_0^{1} \sup_{|y|<t} \| \oD_{y} (s\Ls)^{m} \e^{-s\Ls} f\|_p \, \frac{\dd s}{s} \\
& \lesssim  \int_0^{1} \min(1, s^{-1/2}t)\| (s\Ls)^{m} \e^{-s\Ls} f\|_p \, \frac{\dd s}{s}.
\end{align*}
Then, 
\begin{multline*}
\bigg(\int_0^1  \bigg(t^{-\alpha}  \sup_{|y|<t} \bigg\|\oD_{y} \int_0^{1} (s\Ls)^{m} \e^{-s\Ls} f \, \frac{\dd s}{s}   \bigg\|_p \bigg)^q \, \frac{\dd t}{t}\bigg)^{1/q} \\
 \lesssim  \bigg(\int_0^1 \bigg( \int_0^{1} t^{-\alpha} \min(1, s^{-1/2}t)\| (s\Ls)^{m} \e^{-s\Ls} f\|_p \, \frac{\dd s}{s}\bigg)^q \, \frac{\dd t}{t}\bigg)^{1/q} .
\end{multline*}
Since 
\[
 \int_0^1 t^{-\alpha} s^{\alpha/2} \min(1, s^{-1/2}t) \, \frac{\dd s}{s}	\lesssim 1, \qquad   \int_0^1 t^{-\alpha} s^{\alpha/2} \min(1, s^{-1/2}t) \, \frac{\dd t}{t}\lesssim 1,
\]
uniformly for $t,s \in (0,1)$, respectively, by Schur's test 
\begin{multline*}
\bigg(\int_0^1  \bigg( \int_0^{1} t^{-\alpha} \min(1, s^{-1/2}t)\| (s\Ls)^{m} \e^{-s\Ls} f\|_p \, \frac{\dd s}{s}\bigg)^q \, \frac{\dd t}{t}\bigg)^{1/q}\\
\lesssim \bigg(\int_0^1 (s^{-\alpha/2}  \| (s\Ls)^m \e^{-s\Ls} f \|_p )^q \, \frac{\dd s}{s}\bigg)^{1/q},
\end{multline*}
and this completes the proof.      
\end{proof}

\subsection{Localized norms}
The aim of this subsection is twofold: on the one hand, we shall show an analogue of Proposition~\ref{propsecondirTL} for Besov spaces, but in a \emph{necessarily} weaker form; and this sheds some light on why the Besov case is more involved than the Triebel--Lizorkin case. On the other hand, it will provide us with a useful result, namely Corollary~\ref{corppp} below, which we shall need to prove Theorem~\ref{teoB}. Let us give the following
definition. 

\begin{definition}\label{def:Xpqr}
Suppose $p,q,r \in [1, \infty]$ and $\alpha>0$.  We denote by $X_{\alpha}^{p,q,r}$ the collection of all $ f\in \mathcal{S}'$ such that
\[
\|f\|_{X_\alpha^{p,q,r}} = \bigg(\sum_{n \in \N}\| \eta_{n}f\|_{X_\alpha^{p,q}}^r \bigg)^{1/r} <\infty 
\]
with the usual modification in case $r=\infty$.
\end{definition}
In view of Definition~\ref{def:Xpqr}, Proposition~\ref{propsecondirTL} can be rephrased by saying that $F_{\alpha}^{p,q} = F^{p,q,p}_{\alpha}$ with equivalence of norms when $p,q\in (1,\infty)$. In particular, since $B^{p,p}_\alpha = F^{p,p}_\alpha$, we also get that $B^{p,p}_\alpha = B^{p,p,p}_{\alpha} $ for $p\in (1,\infty)$. As we shall see in Corollary~\ref{corppp} below, this actually holds for $p\in [1,\infty]$, but the situation for general $p,q,r$ is quite different. Indeed, we have only the following results which, on $\R^d$, are ``if and only if''; cf.~\cite[Proposition 3.6]{NS}).

\begin{proposition} \label{proppqr}
Suppose $p,q,r \in [1,\infty]$ and $\alpha>0$.
\begin{itemize}
\item[\emph{(1)}] If $r \leq \min(p,q)$, then $B_\alpha^{p,q,r} \hookrightarrow B_\alpha^{p,q}$;
\item[\emph{(2)}] if $r\geq \max(p,q)$, then $B_\alpha^{p,q}\hookrightarrow B_\alpha^{p,q,r}$.
\end{itemize}
\end{proposition}

\begin{proof}
We shall suppose that $\alpha$ is not an integer, so that $\alpha = k_{0} + \alpha'$ with $k_{0}\in\N$ and $\alpha'\in (0,1)$. When $\alpha$ is an integer, one can argue by interpolation as before. 

We begin by proving~(1). We first recall that~\cite[Theorem 4.5]{BPV-MA}
\begin{equation}\label{recursiveBesov}
\|f\|_{B^{p,q}_\alpha} \approx \sum_{|J|\leq k_{0}} \| X_J f\|_{B^{p,q}_{\alpha'}}.
\end{equation}
Then, by Lemma~\ref{chanineq}
\begin{align}
\|f\|_{B^{p,q}_{\alpha}}\lesssim \sum_{|J|\leq k_{0}}  \bigg( \|X_{J} f\|_p +  \bigg( \sum_{k}  (2^{k\alpha'}  \sup_{|y|<2^{-k}} \| \oD_{y} (X_{J} f)\|_{p} )^q\bigg)^{1/q} \bigg) \label{rpqone}.
\end{align}
Fix $J$ such that $|J|\leq k_{0}$. By arguing as in~\eqref{intp} and since $r\leq p$,
\begin{align*}
\|X_{J}f\|_{p}   &= \Big\|  \sum_{n} X_{J}(f\eta_{n}) \Big\|_p \\
& \lesssim \bigg(  \sum_{n}\| X_{J}(f\eta_{n}) \|_p^{p} \bigg)^{1/p} \lesssim  \bigg(  \sum_{n}\| X_{J}(f\eta_{n})\|_p^{r} \bigg)^{1/r} \lesssim \|f\|_{B_\alpha^{p,q,r}}.
\end{align*}
We consider the second term in~\eqref{rpqone}, which we write as $\sum_{|J|\leq k_{0}} \sigma_{f,J}$. Since $\supp \oD_{y}(X_{J}(\eta_{n}f)) \subseteq x_{n}B_3$, by Lemma~\ref{covering}~(2) arguing as in~\eqref{ci-serve-no?}, we get
\[
|\oD_{y}(X_{J}f)|^{r} \leq \bigg( \sum_{n} |\oD_{y} (X_{J}(\eta_{n}f) )|\bigg)^{r} \lesssim \sum_{n} |\oD_{y}(X_{J}(\eta_{n}f))|^r
\]
with a uniform constant depending only on ($N$ and) $r$. Then, since $p\geq r$,
\begin{equation*}
\begin{split}
\sigma_{f,J}^{r} &= \bigg( \sum_{k}  \Big(2^{k\alpha' r}  \sup_{|y|<2^{-k}} \Big\|\Big( \sum_{n} \oD_{y} (X_{J}(\eta_{n}f) )\Big)^{r}\Big\|_{p/r} \Big)^{q/r}\bigg)^{r/q} \\
& \lesssim \bigg( \sum_{k}  \Big(2^{k\alpha' r}  \sup_{|y|<2^{-k}} \Big\| \sum_{n} |\oD_{y}(X_{J}(\eta_{n}f) )|^r \Big\|_{p/r} \Big)^{q/r}\bigg)^{r/q}\\
&  \lesssim  \bigg( \sum_{k}  (2^{k\alpha' r} \sum_{n}  \sup_{|y|<2^{-k}}\| |\oD_{y}(X_{J}(\eta_{n}f) )|^r \|_{p/r} )^{q/r}\bigg)^{r/q}.
\end{split}
\end{equation*}
It remains to observe that the last quantity equals
\[
\Big\| \Big(\sum_{n} 2^{k\alpha' r}  \sup_{|y|<2^{-k}} \|\oD_{y}(X_{J}(\eta_{n}f) ) \|_{p}^{r} \Big)_{k}\Big\|_{\ell^{q/r}},
\]
so that by the triangle inequality in $\ell^{q/r}$ ($q\geq r$) one gets
\begin{align*}
\sigma_{f,J}^{r} &\lesssim  \sum_{n}  \| ( 2^{k\alpha' r}  \sup_{|y|<2^{-k}} \|\oD_{y}(X_{J}(\eta_{n}f) ) \|_{p}^{r} )_{k}\|_{\ell^{q/r}} \\
&=   \sum_{n}  \bigg(\sum_{k} ( 2^{k\alpha'}  \sup_{|y|<2^{-k}} \|\oD_{y}(X_{J}(\eta_{n}f) ) \|_{p})^{q} \bigg)^{r/q} \\
& \lesssim  \sum_{n} \|X_{J}(\eta_{n}f)\|_{B^{p,q}_{\alpha'}}^{r} \lesssim \| f\|_{B_\alpha^{p,q,r}}^{r}.
\end{align*}
Thus (1) is proved for $r< \infty$. If $r=\infty$,  then $ r=p=q = \infty$. By combining~\eqref{maindec} and~\eqref{equivaderiv}, one gets $\|X_{J}f\|_{\infty} \lesssim \sup_{n} \|X_{J}(f \eta_{n})\|_{\infty}$. Moreover
\begin{align*}
 |\oD_y X_{J}f(x)| & \lesssim  \sum_{n\colon \! x\in x_{n}B(e,3)} |\oD_y (X_{J}(\eta_{n}f))(x)| \lesssim  \sup_{n}   |\oD_y (X_{J}(\eta_{n}f))(x)|,
\end{align*}
so that we conclude
\[
|y|^{-\alpha}|\oD_y (X_{J}f)| \lesssim \sup_{n} \|X_{J}(\eta_{n}f)\|_{B_{\alpha}^{\infty,\infty}} = \|f\|_{B_\alpha^{\infty,\infty,\infty}}
\]
which completes the proof of (1).

\smallskip

We now prove (2). Since again by~\eqref{recursiveBesov}
\begin{align*}
\|f\|_{B_\alpha^{p,q,r}} 
& \approx \bigg(\sum_{n}  \sum_{|I|\leq k_{0}}\|X_{I}(f\eta_{n})\|_{B_{\alpha'}^{p,q}}^{r}\bigg)^{1/r}  \\
& \lesssim \bigg( \sum_{n}  \sum_{|I|+|J|\leq k_{0}}\|X_{I}f \cdot X_{J}\eta_{n}\|_{B_{\alpha'}^{p,q}}^{r}\bigg)^{1/r} \nonumber ,
\end{align*}
it will be enough to show that
\begin{align}\label{frometaftof}
 \Big(\sum_{n } \|X_{I}f \cdot X_{J}\eta_{n}\|_p^{r}\Big)^{1/r} \! &+\bigg( \sum_{n} \bigg(\sum_{k} (2^{k\alpha'} \!\! \sup_{|y|<2^{-k}} \|  \oD_{y} (X_{I}f \cdot X_{J}\eta_{n})\|_{p} )^q\bigg)^{r/q}\bigg)^{1/r}\nonumber \\
& \lesssim \| X_{I}f\|_p +  \bigg( \sum_{k}  (2^{k\alpha'}  \omega_1 (X_{I}f, 2^{-k},p))^q\bigg)^{1/q}
\end{align}
whenever $|I|+|J|\leq k_{0}$. The first term in~\eqref{frometaftof} is easily dealt with: using $r\geq p$,~\eqref{unifbound} and arguing as in~\eqref{maindec}
\begin{align*}
\Big(\sum_{n} \| X_{I}f \cdot X_{J}\eta_{n}\|_{p}^{r}\Big)^{1/r} 
& \leq \Big(\sum_{n} \| X_{I}f \cdot X_{J}\eta_{n}\|_{p}^{p}\Big)^{1/p}\\
&  \lesssim \sup_{n}\|X_{J}\eta_{n}\|_{\infty} \Big(\sum_{n }  \| X_{I}f \cdot \mathbf{1}_{x_{n}B(e,2)}\|_p^{p}\Big)^{1/p}  \lesssim  \| X_{I}f\|_p.
\end{align*}
Next, we consider the second term in the left hand side
of~\eqref{frometaftof}, and we call it $E(f)$. By the triangle
inequality in $\ell^{r/q}$ ($r\ge q$) and then by the embedding
 $\ell^{p} \hookrightarrow\ell^{r}$ ($r\ge p$)
\begin{equation}\label{Af}
\begin{split}
E(f)
&\leq  \bigg(\sum_{k}2^{k\alpha' q}\bigg(\sum_{n}\sup_{|y| < 2^{-k}}\|\oD_y (X_{I}f \cdot X_{J}\eta_{n}) \|_p^r \bigg)^{q/r}\bigg)^{1/q} \\
& \lesssim \bigg(\sum_{k}2^{\alpha' kq}\bigg(\sum_{n}\sup_{|y| < 2^{-k}} \|\oD_y (X_{I}f \cdot X_{J}\eta_{n}) \|_p^p \bigg)^{q/p} \bigg)^{1/q}.
\end{split}
\end{equation}
Recall now that for all $m$
\begin{equation}\label{CRF}
\begin{split}
X_{I}f
&= \frac{1}{(m-1)!} \int_0^{1} (t\Ls)^{m} \e^{-t\Ls} X_{I}f  \frac{\dd t}{t} +  \sum_{\ell=0}^{m-1} \frac{1}{\ell!}  \Ls^\ell \e^{-\Ls} X_{I}f  \\
& =: \sum_{\ell\in \Z } f_{\ell+2k} +  \sum_{\ell=0}^{m-1} \frac{1}{\ell!}  \Ls^\ell \e^{-\Ls} X_{I}f,  
\end{split}
\end{equation}
where $f_\ell=0$ if $\ell\le0$, while if $\ell\ge1$
$$
 f_\ell  =\frac{1}{(m-1)!}  
\int_{2^{-\ell}}^{2^{-\ell+1}} \!\!\!(t\Ls)^{m} \e^{-t\Ls} X_{I}f
\frac{\dd t}{t} . 
$$  
We choose any $m\geq 1$ ($m=1$ would suffice, but we maintain greater generality for later use). Then, $E(f)^{q}\lesssim  I^{q} + I\!I^{q}$,  where
 \begin{align}\label{I} 
  I^q  & =  \sum_{k}2^{\alpha' kq}\bigg(\sum_{n}\sup_{|y| < 2^{-k}} \Big\| \sum_{\ell+2k\geq 0} |\oD_y (X_{J}\eta_n \cdot f_{\ell+2k})|\Big\|_p^p\bigg)^{q/p} ,
   \end{align}
   while
 \begin{align} \label{II}
 I\!I^q &=  \sum_{k}2^{\alpha' kq}\bigg(\sum_{n}\sup_{|y| < 2^{-k}} \sum_{\ell =0}^{m-1}  \|\oD_y (X_{J}\eta_n\cdot \Ls^\ell \e^{-\Ls} X_{I}f ) \|_p^p  \bigg)^{q/p} .  
\end{align}
As for $I\!I$,  since by Lemma~\ref{lemmader}~(3) and~\eqref{unifbound}
\[
\sup_{|y| < 2^{-k}} \|\oD_y ( X_{J}\eta_n\cdot \Ls^\ell \e^{-\Ls}
X_{I}f ) \|_p \lesssim 2^{-k} \|\textbf{1}_{\supp \eta_n}\e^{-c\Ls}|X_{I}f|\|_p,
\]
we obtain, since $m>\alpha'$ and arguing as in~\eqref{maindec}
\[
I\!I \lesssim  \bigg(\sum_{k}2^{\alpha' kq}\Big(\sum_{n} 2^{-kp}
\|\textbf{1}_{\supp \eta_n}\e^{-c\Ls}|X_{I}f|\|_p^p \Big)^{q/p}\bigg)^{1/q}\lesssim \|X_{I}f\|_p.
\]
As for $I$,
 \begin{align*} 
I^{q} & \lesssim \sum_{k}2^{\alpha' kq}\bigg(\sum_{n} \Big( \sum_{{\ell+2k\ge 1}} \sup_{|y| < 2^{-k}}  \| \oD_y (X_{J}\eta_n \cdot f_{\ell+2k})\|_p \Big)^p\bigg)^{q/p}\nonumber\\
  & \lesssim \sum_{k}2^{\alpha' kq}\bigg( \sum_{{ \ell+2k\ge1 }}\Big( \sum_{n}  \sup_{|y| < 2^{-k}} \| \oD_y ( X_{J}\eta_n \cdot f_{\ell+2k}) \|_p^p\Big)^{1/p}\bigg)^{q}.
 \end{align*}
Notice now that for $ \ell + 2k \geq  1 $ and $|y|<2^{-k}$, again by Lemma~\ref{lemmader}~(3) and~\eqref{unifbound}
\begin{equation}\label{newabs}
\begin{split}
\| \oD_y  &( X_{J}\eta_n\cdot  f_{\ell+2k} )\|_p \\
 &\lesssim \int_{2^{-\ell-2k}}^{2^{-\ell-2k+1}}  \big\| \oD_{y} \big( X_{J} \eta_n \cdot \e^{-(t - 2^{-\ell-2k-1}) \Ls} (t\Ls)^{m} \e^{-2^{-\ell-2k-1}\Ls} X_{I}f\big)\big\|_{p}\frac{\dd t}{t}  \\
& \lesssim  2^{\ell /2} \int_{2^{-\ell-2k}}^{2^{-\ell-2k+1}} \big\| \mathbf{1}_{\supp \eta_n} \e^{-c 2^{-\ell-2k} \Ls} |(t\Ls)^{m} \e^{-2^{-\ell-2k-1}\Ls} X_{I}f|\big\|_{p}\frac{\dd t}{t} \\
 &\lesssim  2^{\ell /2}\big\| \mathbf{1}_{\supp \eta_n} \e^{-c 2^{-\ell-2k} \Ls} | ( 2^{-(\ell +2k+1)}\Ls)^{m}\e^{-2^{-\ell-2k-1}\Ls} X_{I}f|\big\|_{p}.
 \end{split}
 \end{equation}
 Therefore, by the bounded overlap property and the $L^{p}$ boundedness of the heat semigroup,
 \begin{align*}
 \sum_{n}  \sup_{|y| < 2^{-k}} & \| \oD_y ( X_{J}\eta_n \cdot f_{\ell+2k}) \|_p^p \\
 & \lesssim  2^{\ell p/2} \big\| \e^{-c 2^{-\ell-2k} \Ls} | ( 2^{-(\ell +2k+1)}\Ls)^{m}\e^{-2^{-\ell-2k-1}\Ls} X_{I}f|\big\|_{p}^{p}\\
 & \lesssim 2^{\ell p/2} \|  ( 2^{-(\ell +2k+1)}\Ls)^{m} \e^{-2^{-\ell-2k-1}\Ls} X_{I}f  \|_p^{p}.
\end{align*}
Moreover,
\begin{align*}
   \Big(\sum_{n}  \sup_{|y| < 2^{-k}} \| \oD_y ( X_{J}\eta_n f_{\ell+2k} )\|_p^p\Big)^{1/p} 
& \lesssim   \Big(\sum_{n} \| X_{J}\eta_n  f_{\ell+2k} \|_p^p\Big)^{1/p} \\
&  \lesssim  \|    f_{\ell+2k} \|_p  \\
 & \lesssim   \|  ( 2^{-(\ell +2k+1)}\Ls)^{m} \e^{-2^{-\ell-2k-1 }\Ls}  X_{I}f \|_p.
 \end{align*}
In other words
\[
  \Big(\sum_{n}  \sup_{|y| < 2^{-k}} \| \oD_y (X_{J}\eta_n f_{\ell+2k} )\|_p^p\Big)^{1/p} \!\! \lesssim  \min (1,  2^{\ell /2} ) \|  ( 2^{-(\ell +2k+1)}\Ls)^{m} \e^{-2^{-\ell-2k -1}\Ls} X_{I}f \|_p.
\]
Hence,
\begin{align*}
I  &\lesssim \bigg( \sum_{k}\bigg(2^{\alpha' k}
\sum_{{ \ell +2k\ge1 }}  \min (1,  2^{\ell /2} ) \|  ( 2^{-(\ell +2k+1)}\Ls)^{m} \e^{-2^{-\ell-2k-1 }\Ls}  X_{I}f  \|_p\bigg)^{q}\bigg)^{1/q},
\end{align*}
and by the triangle inequality in $\ell^q$ we get, as $m>\alpha$,
\begin{align*}
I  & \lesssim   \sum_{\ell\in\Z} \min (1,  2^{\ell /2} )  \left( \sum_{k}2^{\alpha' kq}\|  ( 2^{-(\ell +2k+1)}\Ls)^{m} \e^{-2^{-\ell-2k-1}\Ls}X_{I} f \|_p ^{q}\right)^{1/q} \\
& \lesssim \sum_{\ell\in\Z}  2^{-\frac{\ell \alpha'}{2}}  \min (1,  2^{\ell /2} )   \left( \sum_{k}2^{\alpha' \frac{( \ell+2k+1) q}{2}}
\|  ( 2^{-(\ell +2k+1)}\Ls)^{m} \e^{-2^{-\ell-2k-1 }\Ls}   X_{I}f \|_p ^{q}\right)^{1/q} \\
   & \lesssim \|X_{I}f\|_{B^{p,q}_{\alpha'}}
  \lesssim \|f\|_{B^{p,q}_\alpha},
\end{align*}
which completes the proof.
\end{proof}

\begin{corollary}\label{corppp}
If $p\in [1,\infty]$ and $\alpha>0$, then $B^{p,p}_\alpha = B^{p,p,p}_\alpha$.
\end{corollary}

\begin{remark}\label{Remxi}
As the proof of Proposition~\ref{proppqr} shows, the only properties of $\eta$ which were used were those of all functions in $\sC$.  In addition to this, the proof of part~(2) shows also that for all $p\in [1,\infty]$ and $\xi\in \sC$
\[
\bigg(\sum_{n\in \N} \| f \widetilde \xi_n\|_{B^{p,p}_\alpha}\bigg)^{1/p}  \lesssim \|f\|_{B^{p,p}_\alpha}.
\]
\end{remark}

\subsection{The case $q\geq p$}
We now proceed to proving Theorem~\ref{teoB}.   One implication is given by Proposition~\ref{proponedir}. The other implication is given by the following proposition.

\begin{proposition}\label{prop_embed2}
Suppose $d<p\leq q\leq \infty$ and $\alpha \in (d/p,1)+\N$. Then
\[
\| fg \|_{B_\alpha^{p,q}} \lesssim  \| g \|_{B_\alpha^{p,q}} \|  f \|_{B_\alpha^{p,q}, {\unif}}
\]
for all $g\in B_\alpha^{p,q}$ and $f$ uniformly locally in $B_\alpha^{p,q} $.
\end{proposition}

\begin{proof}  
Let $\xi \in \sC$ be such that  $\xi=1$ on  $\supp \eta$, so that
$\xi \eta = \eta$ and  $\widetilde \xi_n \eta_n=\eta_n$.
By assumption, $\alpha = k_{0}+ \alpha'$ with $d/p<\alpha'<1$ for some $k_{0}\in \N$. We begin by observing that by~\eqref{recursiveBesov}
\begin{align*}
\|fg\|_{B^{p,q}_\alpha} &\approx \sum_{|I|+|J| \leq k_{0}} \|X_{I}f X_{J}g\|_{B^{p,q}_{\alpha'}},
\end{align*}
so that we shall consider, for $|I|+|J| \leq k_{0}$, the quantity
\begin{equation} \label{forrff}
\begin{split}
  & \|X_{I}f X_{J}g\|_{B^{p,q}_{\alpha'}}\\
 & \lesssim   \|X_{I}f  X_{J}g\|_{p} + \bigg( \sum_{k}  (2^{k\alpha'}\!\! \sup_{|y|<2^{-k}}  \|\oD_{y}(X_{I} fX_{J}g)\|_{p})^q\bigg)^{1/q} \\
&    \lesssim \|X_{I}f X_{J}g\|_{p} + \bigg( \sum_{k}
\bigg(2^{k\alpha'} \sup_{|y|<2^{-k}} \bigg\| \sum_{n}
\oD_{y}(\eta_{n}\widetilde\xi_{n} X_{I}fX_{J}g )\bigg\|_{p}\bigg)^q\bigg)^{1/q} ,
\end{split}
\end{equation}
where we have used the identity $\widetilde \xi_n \eta_n=\eta_n$. On the one hand,
\begin{equation}\label{normapIJ}
\begin{split}
\| X_{I}f X_{J}g\|_{p } 
&\lesssim 
 \bigg\| \sum_{n} |\eta_n X_{I}f | | \widetilde\xi_n X_{J}g | \bigg\|_p\\
 & \lesssim \sup_{n}\|\eta_n X_{I}f\|_\infty  \|  X_{J}g  \|_p   \lesssim  \|g\|_{B_\alpha^{p,q}}  \|  f \|_{B_\alpha^{p,q}, {\unif}}  ,
\end{split}
\end{equation}
the last step  thanks to~\eqref{equivaderiv}, Proposition~\ref{propsecondirTL}, and the embedding $B_{\alpha'}^{p,q} \hookrightarrow L^\infty$, cf.~\cite[Theorem 5.1]{BPV-MA}. Moreover, by~\eqref{leibnizD1}
\[
|\oD_y (\eta_{n}X_{I}f\cdot \widetilde\xi_n X_{J}g)| \leq  |(\eta_n X_{I}f) \oD_y (\widetilde\xi_n  X_{J}g)| +|(\widetilde\xi_n X_{J}g)(\cdot y^{-1})\oD_y (\eta_n X_{I}f)|.
 \]
Suppose $|y|\le 1$. Since all the terms appearing in the right hand side above are supported in $x_{n}B_{4}$, as in~\eqref{maindec} we get
\begin{align*}
& \sum_{k}  \bigg(2^{k\alpha'} \sup_{|y|<2^{-k}} \Big\| \sum_{n}
 \oD_y ( \eta_{n}X_{I}f\cdot\widetilde\xi_n X_{J}g)\Big\|_{p}\bigg)^q\\
 &\lesssim   \sum_{k} 2^{k\alpha' q} \sup_{|y|< 2^{-k}} \Big(  \sum_{n}\| (\eta_n X_{I}f) \cdot \oD_y(\widetilde\xi_n  X_{J}g)\|_p^p\Big)^{q/p} \\
& \qquad  +  \sum_{k} 2^{k\alpha' q} \sup_{|y|< 2^{-k}} \Big(  \sum_{n}\| (\widetilde\xi_n X_{J}g)(\cdot y^{-1})\oD_y (\eta_n X_{I}f)\|_p^p\Big)^{q/p} =: \sigma_{0}^{q} + \sigma_{1}^{q}.
\end{align*}
We shall estimate $\sigma_{0}$ and $\sigma_{1}$ separately.

Since $B_{\alpha'}^{p,q} \hookrightarrow L^\infty$, and again by~\eqref{equivaderiv},
\begin{align*}
\|(\eta_n X_{I}f) \oD_y(\widetilde\xi_n  X_{J}g)\|_p 
 & \leq  \|  \oD_y (\widetilde\xi_n  X_{J}g)  \|_p  \| \eta_n X_{I}f \|_\infty \\
 & \lesssim \|  \oD_y (\widetilde\xi_n X_{J}g)  \|_p  \| f \|_{B_\alpha^{p,q }, \unif},
\end{align*}
whence
\[
\sigma_{0}
\lesssim  \bigg(\sum_{k}\bigg(2^{k\alpha' p}\sup_{|y|< 2^{-k}} \sum_{n\in\N}  \|\oD_y(\widetilde\xi_nX_{J}g)  \|_p^p\bigg)^{q/p} \bigg)^{1/q}  \|  f \|_{B_\alpha^{p,q }, \unif}.
\]
As in~\eqref{CRF}, now we write for $m\geq 1$
\[
\widetilde\xi_n X_{J}g = \sum_{\ell \in \Z} \widetilde\xi_n g_{2k+\ell} + \widetilde\xi_n\sum_{\ell=0}^{m-1} \frac{1}{\ell!}  \Ls^\ell \e^{-\Ls} X_{J}g,
\]
which yields
\begin{align*}
& \sum_{k} 
\bigg(2^{k\alpha' p}\sup_{|y|< 2^{-k}} \sum_{n} \|\oD_y(\widetilde\xi_n X_{J}g)  \|_p^p\bigg)^{q/p} 
\\
&\lesssim  \sum_{k} 2^{k\alpha' q}
\bigg(\sum_{n} \sup_{|y|< 2^{-k}}  \bigg\|\sum_{\ell \in \Z}  |\oD_y(\widetilde\xi_n g_{2k+\ell})| + \sum_{\ell=0}^{m-1}|\oD_y( \widetilde\xi_n \Ls^\ell \e^{-\Ls} X_{J}g) | \bigg\|_p^p\bigg)^{q/p} \\
&\lesssim I^{q} + I\! I^{q}
\end{align*}
where $I$ and $I\! I$ are as the ones in~\eqref{I} and~\eqref{II}  with $g$ and $\widetilde\xi_n$ in place of $f$ and $\eta_n$, respectively.
By proceeding exactly as in the proof of Proposition~\ref{proppqr}, we conclude
\begin{align*}
\sigma_{0} &\lesssim  \|g \|_{B_\alpha^{p,q}} \|  f \|_{B_\alpha^{p,q }, \unif}  .
\end{align*}
We now consider $\sigma_{1}$. Since
\begin{align*}
\| (\widetilde \xi_n X_{J}g)(\cdot y^{-j})  \oD_y(\eta_nX_{I}f) \|_p
  & \leq  \| (\widetilde \xi_n X_{J}g) \|_\infty \|\oD_y(\eta_n X_{I}f) \|_p,
\end{align*}
we get
\begin{equation} \label{sigma1}
\sigma_1 \leq \bigg(\sum_{k}\bigg(2^{k\alpha' p}  
\sum_{n} \sup_{|y|< 2^{-k}} \| \widetilde \xi_n X_{J}g  \|_\infty^{p}   \|\oD_y(\eta_n X_{I}f) \|_p^p 
\bigg)^{q/p} \bigg)^{1/q}.
\end{equation}
By the triangle inequality in $\ell^{q/p}$, Lemma~\ref{chanineq} and~\eqref{equivaderiv}, we get
\begin{align*}
\sigma_{1} &  \leq   \bigg(\sum_{n}\| \widetilde \xi_n X_{J}g \|_\infty^p  \bigg(  \sum_{k} 2^{k\alpha' q}     \sup_{|y|< 2^{-k}} \|\oD_y(\eta_n X_{I}f) \|_p^q \bigg)^{p/q} \bigg)^{1/p} \\
&\lesssim  \bigg(\sum_{n}   \| \widetilde \xi_n X_{J}g \|_\infty^p     \|  \eta_n X_{I}f \|_{B_{\alpha'}^{p,q}}^p \bigg)^{1/p} \nonumber \\
& \leq \bigg(\sum_{n}   \| \widetilde \xi_n X_{J}g \|_\infty^p  \bigg)^{1/p}     \| f\|_{B_\alpha^{p,q}, \unif} \nonumber .
\end{align*}
Let now $\varepsilon>0$ be such that $\alpha'-\varepsilon>d/p$. Since $B_{\alpha'-\varepsilon}^{p,p} \hookrightarrow L^\infty$ again by~\cite[Theorem 5.1]{BPV-MA},
\begin{equation}\label{alphaeps}
\bigg( \sum_{n} \| \widetilde \xi_nX_{J}g \|_\infty^p    \bigg)^{1/p} 
\lesssim  \bigg( \sum_{n} \| \widetilde \xi_n X_{J}g \|_{B_{\alpha'-\varepsilon}^{p,p}}^p    \bigg)^{1/p}
\lesssim \| X_{J}g\|_{B_{\alpha'-\varepsilon}^{p,p}}
\end{equation}
the  last bound  by Remark~\ref{Remxi}.   Since  $B_{\alpha'}^{p,q} \hookrightarrow B_{\alpha'-\varepsilon}^{p,p}$ by~\cite[Theorem 5.1]{BPV-MA}, we finally get
\begin{equation*}
\sigma_1 \lesssim  \| g \|_{B_\alpha^{p,q}} \| f\|_{B_\alpha^{p,q}, \unif}
\end{equation*}
and the proof is complete.
\end{proof}

\subsection{The case $B^{1,1}_{d}$ }
Recall that for all $p\in [1,\infty]$, $B^{p,1}_{d/p}$ is an algebra. The condition $q\geq p$ and $q\in [1,\infty]$ restricts to the space $B^{1,1}_{d}$. We have the following proposition, which together with Proposition~\ref{proponedir} concludes the proof of Theorem~\ref{teoB11}.
\begin{proposition}\label{B11d}
For all $g\in B^{1,1}_{d}$ and $f \in B_{d,\unif}^{1,1}$
\[
\|fg\|_{B_{d}^{1,1}}\lesssim \| g\|_{B_{d}^{1,1}} \| f \|_{B_{d}^{1,1}, \unif}.
\]
\end{proposition}

\begin{proof}
Let $\xi \in \sC$ be such that  $\xi=1$ on  $\supp \eta$, so that
$\xi \eta = \eta$ and  $\widetilde \xi_n \eta_n=\eta_n$. By Corollary~\ref{corppp}, the algebra property of $B_{d}^{1,1}$, and Remark~\ref{Remxi},
\begin{align*}
\|fg\|_{B_{d}^{1,1}}
& \lesssim  \sum_{n} \|(\eta_{n} f )(\widetilde\xi_{n} g)\|_{B_{d}^{1,1}} \\
& \lesssim    \sum_{n} \| \eta_n f\|_{B_{d}^{1,1}} \| \widetilde\xi_{n} g\|_{B_{d}^{1,1}}\\
& \lesssim   \| f \|_{B_{d}^{1,1}, \unif}  \sum_{n} \|\widetilde\xi_{n} g\|_{B_{d}^{1,1}} \lesssim  \| f \|_{B_{d}^{1,1}, \unif} \|g\|_{B_{d}^{1,1}},
\end{align*}
and the proof is complete.
\end{proof}

\subsection{The case $q<p$}
We begin with a lemma.

\begin{lemma}\label{embeddingschain}
Suppose $p,q\in [1, \infty]$ and $\alpha \in (0,1)$. Then $ B^{p,q}_\alpha\hookrightarrow M_{\alpha}^{{p,q}}\hookrightarrow  B^{p,q}_{\alpha, \unif}$.
\end{lemma}

\begin{proof}
Pick $n\in\N$ and choose  the sequence $(\gamma_{k}) = \mathbf{1}_{\{k=n\}}$. Then $\sum_{k}\gamma_{k}\eta_{k}f =  \eta_{n}f$, and the second embedding follows.

To prove the first, observe that since $\ell^{p} \hookrightarrow\ell^{\infty}$ and by arguing as in~\eqref{intp} (in one direction), $\| f\|_{M^{p,q}_\alpha}$ is bounded by
\begin{align*}
 \sup_{\| \gamma_{n}\|_{\ell^{p}}\leq 1} & \bigg(\Big( \sum_{n} \|\gamma_{n}\eta_{n}f\|_{p}^{p}\Big)^{1/p} + \bigg(\!\sum_{k}\Big(2^{k\alpha p} \!\sup_{|y| < 2^{-k}} \Big\| \sum_{n}  \oD_y(\gamma_{n}  \eta_{n}f) \Big\|_{p}^p \Big)^{q/p} \bigg)^{1/q}\bigg)
\\
&\lesssim  \bigg(\| f\|_{p}+ \bigg(\sum_{k}\Big(2^{k\alpha p} \sum_{n} 
\sup_{|y| < 2^{-k}}   \|\oD_{y}(\eta_{n} f)\|_{p}^p \Big)^{q/p} \bigg)^{1/q}\bigg)
\end{align*}
so that arguing as from~\eqref{Af} on (with no derivatives) one gets $\| f\|_{M^{p,q}_\alpha} \lesssim \|f\|_{B^{p,q}_\alpha}$, and the first embedding follows.
\end{proof}

\begin{proof}[Proof of Theorem~\ref{q<p}] 
Suppose $\alpha \in (d/p,1)$. Following~\eqref{forrff} (with no derivatives) we shall prove that given $f\in M^{p,q}_\alpha$ and $g\in B^{p,q}_\alpha$
\begin{equation}\label{unodidueD}
\|f g\|_{p} + \bigg( \sum_{k}  (2^{k\alpha} \sup_{|y|<2^{-k}}  \|\oD_{y}(fg)\|_{p})^q\bigg)^{1/q}  \lesssim  \| f\|_{M^{p,q}_\alpha} \| g\|_{B^{p,q}_\alpha}.
\end{equation}
This implies that $\| f\|_{M^{p,q}_\alpha} \gtrsim \|f\|_{MB^{p,q}_\alpha}$.

By~\eqref{normapIJ},
\[
\|f   g\|_{p} \lesssim  \|g\|_{B_\alpha^{p,q}}  \|  f \|_{B_\alpha^{p,q}, {\unif}}\lesssim   \|g\|_{B_\alpha^{p,q}}  \|  f \|_{M^{p,q}_\alpha} ,
\]
the last inequality by Lemma~\ref{embeddingschain}. By arguing as in
the first part of the proof of Proposition~\ref{prop_embed2}  (we
maintain the notation therein)
\[
\bigg( \sum_{k}  (2^{k\alpha} \sup_{|y|<2^{-k}}  \|\oD_{y}( f g)\|_{p})^q\bigg)^{1/q}  \lesssim \sigma_{0}+\sigma_{1}.
\]
Since to estimate $\sigma_{0}$ we did not use any condition on $p$ and $q$, we might argue in the same manner and get
\[
\sigma_{0} \lesssim  \|g\|_{B_\alpha^{p,q}}  \|  f \|_{B_\alpha^{p,q }, \unif} \lesssim \|g\|_{B_\alpha^{p,q}}  \|  f \|_{M^{p,q}_\alpha}
\]
again by Lemma~\ref{embeddingschain}.

We are left with considering $\sigma_{1}$. Select, by Lemma~\ref{covering}, $N$ disjoint families of indices $I_{k}$, $k=1,\dots, N$ with the property that
\[
\N =\bigcup_{k=1}^{N} I_{k}, \qquad  d_C(x_{m},x_{h}) \geq 6 \quad \forall m,h \in I_{k}, \, m\neq k, \; \forall \, k.
\] 
Then, for $|y|\leq 1$ and $m,h \in I_{k}$, $m\neq h$,
\[
\supp \oD_{y}(\eta_{m}f)\cap  \supp \oD_{y}(\eta_{h}f) \subseteq  x_{m}B_3 \cap x_{h}B_3  = \emptyset
\]
and thus
\begin{equation}\label{decIk}
\begin{split}
\sum_{n \in I_{k}}|\gamma_{n}|^{p}\|\oD_{y}(\eta_{n}f)\|_{p}^{p} 
&= \sum_{n \in I_{k}} \int_{G}|\gamma_{n}|^{p} |\oD_{y}(\eta_{n}f)|^{p} \, \dd\lambda \\
& = \int_{G}\Big| \sum_{n \in I_{k}} \oD_{y}(\gamma_{n}\eta_{n}f)\Big|^{p}\, \dd \lambda = \Big\| \sum_{n \in I_{k}} \oD_{y}(\gamma_{n}\eta_{n}f)\Big\|_{p}^{p}.
\end{split}
\end{equation}
Now fix $k=1,\dots, N$ and pick the sequence (we assume  $g\neq 0$ here)
\[
\gamma_n = \gamma \frac{ \|\eta_{n} g\|_{\infty} }{ \| g\|_{B^{p,q}_{\alpha'}}} \quad \mbox{if } n\in I_{k}, \qquad \gamma_{n}=0 \quad \mbox{otherwise}.
\]
If $\gamma$ is small enough, then the sequence $(\gamma_n)$ is in $\ell^{p}$ and has $\ell^{p}$ norm smaller than $1$; recall~\eqref{alphaeps}. This also shows that $\gamma$ can be chosen independent of $g$.

Therefore, by~\eqref{sigma1}, the embedding $B^{p,q}_{\alpha} \subseteq L^{\infty}$ and~\eqref{decIk}
\begin{align*}
\sigma_1 &\lesssim  \bigg(\sum_{k}\bigg(2^{k\alpha p} \sup_{|y|< 2^{-k}}  
\sum_{n} \| \widetilde\xi_n g  \|_\infty^{p}   \|\oD_y(\eta_n f) \|_p^p 
\bigg)^{q/p} \bigg)^{1/q} \\
& \lesssim  \|g\|_{B^{p,q}_\alpha} \bigg(\sum_{k}\bigg(2^{k\alpha p} \sup_{|y| < 2^{-k}}  
\Big\| \sum_{n} \oD_y(\gamma_{n}\eta_n f) \Big\|_p^p \bigg)^{q/p} \bigg)^{1/q}  \lesssim \|g\|_{B^{p,q}_\alpha}\| f\|_{M^{p,q}_\alpha} 
\end{align*}
which concludes the proof of~\eqref{unodidueD}.

Suppose now that $f \in  M B^{p,q}_\alpha$  and $(\gamma_n)\in \ell^{p}$. Then~\eqref{decIk} implies
\begin{equation}\label{normafuori}
\|f\|_{M_{\alpha}^{p,q}}  \lesssim  \sum_{k=1}^{N}\sup_{\| \gamma_{n}\|_{\ell^{p}}\leq 1} \Big\| \sum_{n \in I_{k}} \gamma_{n}\eta_{n}f\Big\|_{B^{p,q}_{\alpha}}.
\end{equation}
By the algebra property of $B^{p,q}_{\alpha}$,
\begin{equation}\label{MBD}
\Big\| \sum_{n \in I_{k}}\gamma_{n} \eta_{n}f\Big\|_{B^{p,q}_{\alpha}} \lesssim \|f\|_{MB^{p,q}_{\alpha}} \Big\| \sum_{n \in I_{k}}\gamma_{n}\eta_{n}\Big\|_{B^{p,q}_{\alpha}},
\end{equation}
where by~\eqref{unifbound}
\begin{align*}
\Big\| \sum_{n \in I_{k}} & \gamma_{n}\eta_{n}\Big\|_{B^{p,q}_{\alpha}} 
 \lesssim  \Big\|   \sum_{n \in I_{k}} \gamma_{n}\eta_{n}\Big\|_{p} + \bigg(\sum_{k} 2^{k\alpha q} \sup_{|y| < 2^{-k}}  \bigg(\sum_{n}
\|\oD_{y}( \gamma_{n}\eta_n)\|_{p}^p \bigg)^{q/p} \bigg)^{1/q} \\
&\lesssim \Big(   \sum_{n}\| \gamma_{n}\eta_{n}\|_{p}^p\Big)^{1/p} + \bigg(\sum_{k} 2^{k\alpha q} \sup_{|y| < 2^{-k}}  \bigg(\sum_{n}|\gamma_{n}|^{p} \|\oD_{y}\eta \|_{p}^p \bigg)^{q/p} \bigg)^{1/q} \\
& \lesssim  \|(\gamma_{n})\|_{\ell^{p}} \| \eta\|_{B^{p,q}_\alpha} \lesssim 1.
\end{align*}
By~\eqref{normafuori} and~\eqref{MBD} we conclude
\begin{align*}
\|f\|_{M_{\alpha}^{p,q}} \lesssim  \|f\|_{MB^{p,q}_{\alpha}}
\end{align*}
and this completes the proof.
\end{proof}
The case $p=\infty$ was excluded by Theorem~\ref{q<p}, but it is easier as the following shows.
\begin{theorem}
Suppose $\alpha>0$ and $q\in [1,\infty]$. Then $M B^{\infty,q}_{\alpha}=B^{\infty,q}_{\alpha}$ with equivalence of norms.
\end{theorem}
\begin{proof}
One the one hand, $B^{\infty,q}_{\alpha} \hookrightarrow M B^{\infty,q}_{\alpha}$ by the algebra property of $B^{\infty,q}_{\alpha}$. On the other hand, since the constant function equal to $1$ belongs to $B^{\infty,q}_{\alpha}$, one also gets $M B^{\infty,q}_{\alpha} \hookrightarrow B^{\infty,q}_{\alpha}$.
\end{proof}

\section{Stratified groups and wider ranges}\label{Sec:stratified}
In this section we assume that $G$ is a stratified group with the standard dilations $\delta_{s}$, $s>0$, and $\bf{X}$ is a basis of the first layer of the Lie algebra $\mathfrak{g}$. We recall that $G$ is said to be {\em stratified} if its Lie algebra $\mathfrak{g}$ admits a stratification
\[
\mathfrak{g} = V_1\oplus \cdots \oplus V_S,
\]
where $V_1=\operatorname{span} \mathbf X$, $V_{j+1}=[V_1,V_j]$, for $j=1,\dots,S-1$,  and $[V_1,V_S]=0$. We refer the reader to~\cite{Folland, FS} for the basic facts on stratified groups.

In order not to cause any confusion, we shall stress that $G$ is a
stratified Lie group in all the important statements of the section. We shall extend all the theorems in the previous section to the case of all (allowed) regularities.

\subsection{The case $q\geq p$}
\begin{theorem}\label{teoBstratified}
Let $G$ be a stratified group. Suppose $p,q\in [1,\infty]$, $q\geq p$ and $\alpha > d/p$. Then $MB^{p,q}_\alpha = B^{p,q}_{\alpha,\unif}$ with equivalence of norms.
\end{theorem}
Inspired by~\cite{Giulini, GiuliniH}, for $m, \theta \in \N$ and $x,y\in G$, we define
\begin{equation}\label{Gythetam}
\oG_{y,\theta}^{(m)} f(x) = \sum_{\ell =0}^{m} (-1)^{m -\ell}  \binom{m}{\ell}  f(x \delta_{\ell+\theta}(y^{-1}) ) .
\end{equation}
When $\theta=0$, we shall simply write $\oG^{(m)}_{y}$ for $\oG_{y,0}^{(m)}$. Observe that $\oG_{y}^{(1)}=\oD_{y} $.

The following identities hold true: for all $m, \theta\in \N$,
\[
\oG^{(m)}_{y,\theta+1} = \oG^{(m+1)}_{y,\theta} + \oG^{(m)}_{y,\theta},
\]
from which one gets 
\begin{equation}\label{theta-0}
\oG^{(m)}_{y,\theta} = \sum_{j=0}^{\theta} \binom{\theta}{j} \oG^{(m+j)}_{y},
\end{equation}
and moreover one has the ``Leibniz'' rule
\begin{equation}\label{leibnizG}
\oG^{(m)}_{y,\theta} (fg) = \sum_{j=0}^{m} \binom{m}{j} \oG^{(j)}_{y,\theta}f \cdot \oG^{(m-j)}_{y,\theta+j}g.
\end{equation}
All the above identities can be proved by induction, and we omit the details. The following lemma is the high-order counterpart of Lemma~\ref{lemmader}.
\begin{lemma}\label{lemmaderstrat2}
Suppose $m, \theta\in \N$,  $p\in [1,\infty]$ and $y\in B_1$. Then the following holds.
\begin{itemize}
\item[{\rm (1)}] $\|\oG_{y,\theta}^{(m)} f \|_p \lesssim \| f\|_p$;
\item[{\rm (2)}] $ \|\oG_{y,\theta}^{(m)} f \|_p \lesssim |y|^m  \sum_{|J|\leq m} \|X_J f\|_p$;
\item[{\rm (3)}]  for all $k\in \N$ and $\psi \in C_c^{\infty}$ there exist $c=c(k)>0$ and $C(\psi)>0$ such that for all $t\in (0,1)$
\[
 \| \oG_y^{(m)} (\psi \Ls^{k}\e^{-t \Ls} f)\|_p \leq C(\psi) t^{-\frac{m}{2}-k} |y|^{m} \|\mathbf{1}_{\supp \psi}\, \e^{-ct \Ls} |f| \|_p,
\]
where $C(\psi)$ depends only on $\|\psi\|_{L^{\infty}_{m}}$.
\end{itemize}
\end{lemma}

\begin{proof}
Statement (1) is obvious. Statement (2) has been proven in~\cite[Proposition 1]{Giulini} when $\theta=0$. We outline its proof following ~\cite[Lemma 2]{GiuliniH}  without giving all the details. Since $G$ is stratified, given $y\in B_1$, there exist $v_1,\dots,v_M$, such that $y^{-1}=v_1 \dots  v_M$, with $v_i={\rm{exp}}X^{(i)}$, $X^{(i)}\in V_1$, $|v_i|\lesssim |y|$, $i=1,\dots,M$ (see \cite[Lemma (1.40)]{FS}).  For every $x\in G$, $\theta\in\mathbb N$, $\ell=0,\dots,m,$ we write
$$
f\big(x\delta_{\ell+\theta}(y^{-1})\big)=\sum_{i=1}^M\big[f\big(x\delta_{\ell+\theta}(v_1\dots v_{i-1})\delta_{\ell+\theta}(v_i)\big) - f\big(x\delta_{\ell+\theta}(v_1\dots v_{i-1})\big)\big]+f(x).
$$
Notice that for every $z\in G$, $v\in {\rm{exp}} (V_1)$, $j\in\mathbb N$ 
$$
\frac{d}{ds} f(z\delta_s(v))_{\vert s=j}=E(v)f (z\delta_j(v)),
$$
where $E(v)=\sum_{j=1}^{\kappa}c_j(v)X_j$. Taylor's formula applied to the function $s\mapsto f(z\delta_s(v))$ shows that 
$$
\begin{aligned}
&f\big(x\delta_{\ell+\theta}(v_1\dots v_{i-1})\delta_{\ell+\theta}(v_i)\big) - f\big(x\delta_{\ell+\theta}(v_1\dots v_{i-1})\big)\\&=\sum_{k=0}^{m-1}\frac{(\ell+\theta)^k}{k!} [E(v_i)^k f]\big(x\delta_{\ell+\theta}(v_1\dots v_{i-1})\big)\\
&\qquad +\frac{(\ell+\theta)^m}{(m-1)!}\int_0^1(1-s)^{m-1}[E(v_i)^mf]\big(x\delta_{\ell+\theta}(v_1\dots v_{i-1})\delta_{s}(v_i)\big)\, ds.
\end{aligned}
$$
By the previous equality, arguing as in \cite[Lemma 2]{GiuliniH} for every $\ell=0,\dots,m,$ we can write 
$$
\begin{aligned}
f\big(x\delta_{\ell+\theta}(y^{-1})\big)
&= \sum_{n=0}^{m-1}(\ell+\theta)^{n}Q_{n}(x,y^{-1})+R(x,y^{-1},\ell+\theta)
\end{aligned}
$$
for suitable functions $Q_n$ and remainder terms $R$. It follows that 
$$
\begin{aligned}
\oG_{y,\theta}^{(m)}
f(x)&=\sum_{n=0}^{m-1}\sum_{\ell=0}^m(-1)^{m-\ell}\binom{m}{\ell}
(\ell+\theta)^{n} Q_{n}(x,y^{-1})\\ &\qquad +\sum_{\ell=0}^m(-1)^{m-\ell}\binom{m}{\ell}  R(x,y^{-1},\ell+\theta)\\
&=\sum_{\ell=0}^m(-1)^{m-\ell}\binom{m}{\ell}  R(x,y^{-1},\ell+\theta),
\end{aligned}
$$
where we used the fact that $\sum_{\ell=0}^m(-1)^{\ell}\binom{m}{\ell} \ell^k=0$ for every $k<m$, and where $R(x,y^{-1},\ell+\theta)$ is a linear combination of terms of the form
$$
(\ell+\theta)^m\int_0^1(1-s)^{m-i}Df(xu(s))ds,
$$
with $|u(s)|\lesssim |y|$, $D=\sum_{|J|\leq m}  c^D_JX_J$, and $i=0,\dots,m-1$. It follows that 
$$
\begin{aligned}
\|\oG_{y,\theta}^{(m)} f\|_p&\lesssim \sum_{\ell=0}^m \| R(\cdot,y^{-1},\ell+\theta)\|_p\\
&\lesssim (\ell+\theta)^m|y|^m\sum_{|J|\leq m}  \|X_Jf\|_p\lesssim |y|^m\sum_{|J|\leq m}  \|X_Jf\|_p,
\end{aligned}
$$
as required in (2).

Statement (3) can be proved as in Lemma~\ref{lemmader} by means of (2).
\end{proof}

Define
\[
 \omega_{m} (f,t,p) = \sup_{|y|<t} \|\oG_{y}^{(m)} f \|_p \qquad t>0.
\]
For $p,q\in [1,\infty]$, $\alpha>0$ and $m>\alpha$ we have the equivalences of norms 
\begin{equation*}
\begin{split}\label{Besovequivnorm}
\|f\|_{B^{p,q}_{\alpha}} &\approx \|f\|_p + \bigg(\int_0^1 (t^{-\alpha}  \sup_{|y|<t} \|\oG_{y}^{(m)} f \|_p )^q \, \frac{\dd t}{t}\bigg)^{1/q}  \\
&\approx  \|f  \|_p + \bigg( \sum_{k\in \N}  (2^{k\alpha}  \omega_{m} (f, 2^{-k},p))^q\bigg)^{1/q}.
\end{split}
\end{equation*}
The first can be proved by putting together~\cite[Proposition 5.2]{FMV} and~\cite[Proposition 4]{Giulini}. The second is just a discretization as in Lemma~\ref{equivfirst}. 

We now proceed to proving Theorem~\ref{teoBstratified}.  One implication is given by Proposition~\ref{proponedir}. The other implication is the following proposition.

\begin{proposition}\label{prop_embed3}
Let $G$ be a stratified group. Suppose $p, q\in [1,\infty]$, $q\geq p$ and $\alpha>d/p$. Then
\[
\| fg \|_{B_\alpha^{p,q}} \lesssim  \| g \|_{B_\alpha^{p,q}} \|  f \|_{B_\alpha^{p,q}, {\unif}}
\]
for all $g\in B_\alpha^{p,q}$ and $f$ uniformly locally in $B_\alpha^{p,q} $.
\end{proposition}

\begin{proof} 
Let $\xi \in \sC$ be such that $\xi = 1$ on $\supp \eta$, so that $\xi \eta = \eta$, and $\eta_n=\eta_n\widetilde\xi_n$. Then, for $m>\alpha$
\begin{align*}
\|fg\|_{B^{p,q}_\alpha}  \approx \|fg\|_{p} +  \bigg( \sum_{k\in \N}  \bigg(2^{k\alpha} \sup_{|y|<2^{-k}} \Big\| \sum_{n} \oG_{y}^{(2m)}(fg \eta_{n}\widetilde\xi_{n})\Big\|_{p}\bigg)^q\bigg)^{1/q}.
\end{align*}
First we observe that 
\begin{align}\label{fgpstrat}
\| f g\|_{p } 
&\lesssim 
 \Big\| \sum_{n\in \N} |\widetilde\xi_n g | | \eta_n   f | \Big\|_p \lesssim  \|   g  \|_p  \sup_{n}\|\eta_n  f\|_\infty \lesssim  \|g\|_{B_\alpha^{p,q}}  \|  f \|_{B_\alpha^{p,q}, {\unif}}  ,
\end{align}
the last step by the embedding $B_{\alpha}^{p,q} \hookrightarrow L^\infty$.  Then, by~\eqref{leibnizG}
\begin{align*}
  |\oG_{y}^{(2m)} (fg)|
  & \leq \sum_{n} |\oG_{y}^{(2m)} (\eta_n f g)| = \sum_{n}
    |\oG_{y}^{(2m)} (\eta_n f \widetilde\xi_n g)| 
   \\
  &    \leq \sum_{j=0}^{2m} \binom{2m}{j} \sum_{n}
|\oG_{y,j}^{(2m-j)}( \widetilde\xi_n g)\oG_{y}^{(j)}( \eta_n
    f)|.
\end{align*}
Suppose $|y|\le 1$. Since $ \oG_{y,j}^{(2m-j)} (\widetilde\xi_n g)$ and $\oG_{y}^{(j)} (\eta_n  f)$ are supported in $B_{2m+3}$, and  $\ell^{p} \hookrightarrow \ell^{q}$
we have 
\begin{align*}
  \bigg( \sum_{k}
  &\bigg(2^{k\alpha} \sup_{|y|<2^{-k}} \bigg\| \sum_{n}
    \oG_{y}^{(m)}(\eta_{n}f\widetilde\xi_n g )\bigg\|_{p}\bigg)^q\bigg)^{1/q} \\
 &\lesssim \sum_{j=0}^{2m} \bigg(\sum_{k} 2^{k\alpha q} \sup_{|y|< 2^{-k}}  
\bigg\|\sum_{n} \oG_{y,j}^{(2m-j)}(\widetilde\xi_n g) \oG_{y}^{(j)}(\eta_n f)\bigg\|_p^q \bigg)^{1/q}\\
&\lesssim  \sum_{j=0}^{2m}  \bigg(\sum_{k}  2^{k\alpha q} \sup_{|y|< 2^{-k}}  \bigg(  \sum_{n}\|\oG_{y,j}^{(2m-j)}(\widetilde\xi_n g)\oG_{y}^{(j)}(\eta_n f)\|_p^p\bigg)^{q/p} \bigg)^{1/q}\\
&=: \sum_{j=0}^{2m} \sigma_j.
\end{align*}
To  estimate this last term, we separate the cases when $j\leq m$ and $j>m$.

Suppose first $0 \le j \leq  m$. Since $B_\alpha^{p,q} \hookrightarrow L^\infty$,
\begin{align*}
\|\oG_{y,j}^{(2m-j)}(\widetilde\xi_ng)  \oG_{y}^{(j)}(\eta_n f) \|_p 
 & \leq  \|  \oG_{y,j}^{(2m-j)}(\widetilde\xi_n g) \|_p  \|\oG_{y}^{(j)}(\eta_n f) \|_\infty \\
 & \lesssim  \|  \oG_{y,j}^{(2m-j)}(\widetilde\xi_n g) \|_p  \| \eta_n f   \|_\infty\\
 & \lesssim \|  \oG_{y,j}^{(2m-j)}(\widetilde\xi_n g)  \|_p  \|  f \|_{B_\alpha^{p,q }, \unif}.
\end{align*}
 Thus
\[
\sum_{j=0}^{m} \sigma_j 
\lesssim  \sum_{j=0}^{m}  \bigg(\sum_{k=0}^{\infty}\bigg(2^{k\alpha p}\sup_{|y|< 2^{-k}} \sum_{n} 
\|\oG_{y,j}^{(2m-j)}(\widetilde\xi_n g)  \|_p^p\bigg)^{q/p} \bigg)^{1/q}  \|  f \|_{B_\alpha^{p,q }, \unif}.
\]
As in~\eqref{CRF}, we write (with $2m-j$ in place of $m$ and with no derivative)
\[
 g_\ell  =\frac{1}{(2m-j-1)!}  
\int_{2^{-\ell}}^{2^{-\ell+1}} \!\!\!(t\Ls)^{2m-j} \e^{-t\Ls} f \, \frac{\dd t}{t}, \qquad \ell\ge1,
\] 
and $g_\ell=0$ if $\ell\le0$, so that
\[
\widetilde\xi_n g = \sum_{\ell \in \Z} \widetilde\xi_n g_{2k+\ell} + \widetilde\xi_n\sum_{\ell=0}^{m} \frac{1}{\ell!}  \Ls^\ell \e^{-\Ls} g.
\]
This yields
\begin{align}
&\bigg(\sum_{k} 
\bigg(2^{k\alpha p}\sup_{|y|< 2^{-k}} \sum_{n} \|\oG_{y,j}^{(2m-j)}(\widetilde\xi_n g)  \|_p^p\bigg)^{q/p} \bigg)^{1/q} \label{thison}
 \lesssim  I + I\!I ,
\end{align}
 where 
  \begin{align*} 
  I^q  & =  \sum_{k}2^{\alpha kq}\bigg(\sum_{n}\sup_{|y| < 2^{-k}} \Big\| \sum_\ell |\oG_{y,j}^{(2m-j)} (\widetilde\xi_n g_{\ell+2k})|\Big\|_p^p\bigg)^{q/p} \nonumber\\
 & \lesssim \sum_{k}2^{\alpha kq}\bigg(\sum_{n} \Big( \sum_\ell \sup_{|y| < 2^{-k}}  \| \oG_{y,j}^{(2m-j)} (\widetilde\xi_ng_{\ell+2k})\|_p \Big)^p\bigg)^{q/p}\nonumber\\
  & \lesssim \sum_{k}2^{\alpha kq}\bigg( \sum_\ell \Big( \sum_{n}  \sup_{|y| < 2^{-k}} \| \oG_{y,j}^{(2m-j)} ( \widetilde\xi_n g_{\ell+2k}) \|_p^p\Big)^{1/p}\bigg)^{q},
 \end{align*}
while
 \begin{align*} 
 I\!I^q &=  \sum_{k}2^{\alpha kq}\bigg(\sum_{n}\sup_{|y| < 2^{-k}} \sum_{\ell =0}^{2m-j-1}  \|\oG_{y,j}^{(2m-j)} (\widetilde\xi_n \Ls^\ell \e^{-\Ls} g ) \|_p^p  \bigg)^{q/p} .  
\end{align*}
As for $I\!I$, by~\eqref{theta-0} and Lemma~\ref{lemmaderstrat2}~(3) with~\eqref{unifbound}
\begin{align*}
\sup_{|y| < 2^{-k}} \|\oG_{y,j}^{(2m-j)} ( \widetilde\xi_n\Ls^\ell \e^{-\Ls} g) \|_p 
&\lesssim  \sum_{{h=0}}^{j} \sup_{|y| < 2^{-k}}
\|\oG^{{(2m-j+h)}}_{y}
( \widetilde\xi_n\Ls^\ell \e^{-\Ls} g) \|_{p} \\
&\lesssim 2^{-k(2m-j)}\|\mathbf{1}_{\supp \widetilde\xi_n} \e^{-c\Ls}|g||\|_p, 
\end{align*}
whence we obtain, since $2m-j\geq m >\alpha$,
\[
I\!I \lesssim  \bigg(\sum_{k}2^{\alpha kq}\Big(\sum_{n} 2^{-kmp} \|\mathbf{1}_{\supp \widetilde\xi_n} \e^{-c\Ls}|g|\|_p^p \Big)^{q/p}\bigg)^{1/q}\lesssim \|g\|_p.
\]
As for $I$, when $\ell + 2k \geq 1$ we observe that again by~\eqref{theta-0} and Lemma~\ref{lemmaderstrat2}~(3) with~\eqref{unifbound}, as in~\eqref{newabs}, if $|y| < 2^{-k}$ then
\begin{align*}
& \|\oG_{y,j}^{(2m-j)}  ( \widetilde\xi_n g_{\ell+2k}) \|_p\\
 &\lesssim  \sum_{{h=0}}^{j} \int_{2^{-\ell-2k}}^{2^{-\ell-2k+1}}  \big\|  \oG_{y,j}^{(2m-j+h)}  \big(  \widetilde\xi_n \e^{-(t - 2^{-\ell-2k-1}) \Ls} (t\Ls)^{2m-j} \e^{-2^{-\ell-2k-1}\Ls} f\big)\big\|_{p}\frac{\dd t}{t}   \\
& \lesssim   \sum_{{h=0}}^{j} 2^{\ell(2m-j+h)/2} \bigg(\int_{2^{-\ell-2k}}^{2^{-\ell-2k+1}} \big\| \mathbf{1}_{\supp \widetilde \xi_n} \e^{-c 2^{-\ell-2k} \Ls} |(t\Ls)^{2m-j} \e^{-2^{-\ell-2k-1}\Ls}  f|\big\|_{p}\frac{\dd t}{t} \\
 &\lesssim   \sum_{{h=0}}^{j}  2^{\ell(2m-j+h)/2} \big\| \mathbf{1}_{\supp \widetilde \xi_n} \e^{-c 2^{-\ell-2k} \Ls} | ( 2^{-(\ell +2k+1)}\Ls)^{2m-j}\e^{-2^{-\ell-2k-1}\Ls} f|\big\|_{p},
 \end{align*}
whence
\begin{align*}
 \sum_{n}  & \sup_{|y| < 2^{-k}} \| \oG_{y,j}^{(2m-j)}  ( \widetilde\xi_ng_{\ell+2k} )\|_p^p  \\
 & \lesssim \sum_{h=0}^{j}2^{\ell(2m-j+h)p/2}\|  ( 2^{-(\ell +2k+1)}\Ls)^{2m-j} \e^{-2^{-\ell-2k-1 }\Ls} g \|_p^{p},
\end{align*}
as well as
\begin{align*}
   \Big(\sum_{n}  \sup_{|y| < 2^{-k}} \|\oG_{y,j}^{(2m-j)} ( \widetilde\xi_n g_{\ell+2k} )\|_p^p\Big)^{1/p} 
& \lesssim   \Big(\sum_{n} \| \widetilde\xi_n g_{\ell+2k} \|_p^p\Big)^{1/p} \\
&  \lesssim  \|    g_{\ell+2k} \|_p  \\
 & \lesssim   \|  ( 2^{-(\ell +2k)}\Ls)^{2m-j} \e^{-2^{-\ell-2k }\Ls} g \|_p.
 \end{align*}
In other words
$$
\begin{aligned}
 & \Big(\sum_{n}  \sup_{|y| < 2^{-k}} \|\oG_{y,j}^{(2m-j)}
 (\widetilde\xi_n g_{\ell+2k} )\|_p^p\Big)^{1/p} \\ & \qquad \lesssim  \sum_{h=0}^{j} \min (1,  2^{\ell(2m-j+h) /2} )   \|  ( 2^{-(\ell +2k+1)}\Ls)^{2m-j} \e^{-2^{-\ell-2k-1 }\Ls} g \|_p,
\end{aligned}
$$
hence $I  \lesssim \sum_{h=0}^{j} I_{h}$, where
\begin{align*}
I_{h} & \lesssim \!\bigg(\! \sum_{k}\bigg(2^{\alpha k}
\!\sum_{\ell \in \Z}  \min (1,  2^{\ell(2m-j+h) /2} ) \|  ( 2^{-(\ell +2k+1)}\Ls)^{2m-j} \e^{-2^{-\ell-2k-1 }\Ls} g \|_p\bigg)^{q}\bigg)^{\frac{1}{q}}.
\end{align*}
By the triangle inequality in $\ell^q$ we get, as $2m-j>\alpha$,
\begin{align*}
I_{h}  & \lesssim   \sum_\ell \min (1,  2^{\ell(2m-j+h) /2} )  \bigg( \sum_{k}2^{\alpha kq} \|  ( 2^{-(\ell +2k+1)}\Ls)^{2m-j} \e^{-2^{-\ell-2k-1 }\Ls} g\|_p ^{q}\bigg)^{1/q} \\
& =  \sum_\ell  2^{-\ell \alpha/2}  \min (1,  2^{\ell (2m-j+h)/2} )    \\
& \qquad \qquad \times \bigg( \sum_{k}2^{\alpha ( \ell+2k) q/2}
\|  ( 2^{-(\ell +2k+1)}\Ls)^{2m-j} \e^{-2^{-\ell-2k-1 }\Ls} g\|_p ^{q}\bigg)^{1/q} \\
& \lesssim \|g\|_{B^{p,q}_\alpha},
\end{align*}
and the case  $j=0,\dots, m$ is done.

Suppose now that $m < j \leq 2 m$. We have
\begin{align*}
\|\oG_{y,j}^{(2m-j)}(\widetilde\xi_ng)   \oG_{y}^{(j)}(\eta_n f) \|_p 
& \lesssim 
\|\oG_{y,j}^{(2m-j)}(\widetilde\xi_ng)\|_\infty \|\oG_{y}^{(j)}(\eta_n f) \|_p
 \\
 & \lesssim    \| \widetilde\xi_n g  \|_\infty   \|\oG_{y}^{(j)}(\eta_n f) \|_p.
\end{align*}
Therefore, by the triangle inequality in $\ell^{q/p}$,
\begin{align}
\sigma_j &\lesssim \bigg(\sum_{k}\bigg(2^{k\alpha p} \sup_{|y|< 2^{-k}}  
\sum_{n} \| \widetilde\xi_n g  \|_\infty   \|\oG_{y}^{(j)}(\eta_n f) \|_p^p 
\bigg)^{q/p} \bigg)^{1/q} \label{j>m} \nonumber\\
& \lesssim \bigg(\sum_{n} \bigg( \|\widetilde\xi_ng \|_\infty^q  \sum_{k} 2^{k\alpha q}     \sup_{|y|< 2^{-k}} \|\oG_{y}^{(j)}(\eta_n f) \|_p^q \bigg)^{p/q} \bigg)^{1/p} \nonumber 
\\
&\lesssim  \bigg(\sum_{n}   \| \widetilde\xi_n g \|_\infty^p      \|  \eta_n f \|_{B_\alpha^{p,q}}^p \bigg)^{1/p}    \leq \bigg(\sum_{n}   \| \widetilde\xi_n g \|_\infty^p  \bigg)^{1/p}      \| f\|_{B_\alpha^{p,q}, \unif} .
\end{align}
Let now $\varepsilon>0$ be such that $\alpha-\varepsilon>d/p$. Then $B_{\alpha-\varepsilon}^{p,p} \hookrightarrow L^\infty$, hence
\begin{equation}\label{alphamenoeps}
\bigg( \sum_{n} \| \widetilde\xi_n g \|_\infty^p    \bigg)^{1/p} 
\lesssim  \bigg( \sum_{n} \| \widetilde\xi_n  g \|_{B_{\alpha-\varepsilon}^{p,p}}^p    \bigg)^{1/p} 
 \lesssim \| g\|_{B_{\alpha-\varepsilon}^{p,p}}
\end{equation} 
the last bound by Remark~\ref{Remxi}. Since  $B_\alpha^{p,q} \hookrightarrow B_{\alpha-\varepsilon}^{p,p}$, we get
\[
\sigma_j \lesssim  \| g \|_{B_\alpha^{p,q}} \| f \|_{B_\alpha^{p,q}}
\]
also for $m<j\leq 2m$. The proof is complete.
\end{proof}

\subsection{The case $q<p$} We shall prove the analogue of Theorem~\ref{q<p} for all regularities. We begin with the following lemma.
\begin{lemma}\label{embeddingschainstrat}
Suppose $ p,q \in [1, \infty]$ and $\alpha>0$. Then $ B^{p,q}_\alpha\hookrightarrow M_{\alpha}^{{p,q}}\hookrightarrow  B^{p,q}_{\alpha, \unif}$.
\end{lemma}

\begin{proof}
Pick $n\in\N$ and choose  the sequence $\gamma_{k} = \mathbf{1}_{\{n\}}(k)$. Then $\sum_{k}\gamma_{k}\eta_{k}f =  \eta_{n}f$, and the second embedding follows.

To prove the first, observe that since $(\gamma_{n})\in \ell^{\infty}$
\[
\| f\|_{M^{p,q}_\alpha}\leq 
\|f\|_{p}+ \bigg(\sum_{k}\Big(2^{k\alpha p} 
\sup_{|y| < 2^{-k}} \sum_{n}   \|\oG_y^{(m)}(\eta_{n} f)\|_{p}^p \Big)^{q/p} \bigg)^{1/q}
\]
so that arguing as from~\eqref{thison} on, one gets $\| f\|_{M^{p,q}_\alpha} \lesssim \|f\|_{B^{p,q}_\alpha}$, and the first embedding follows.
\end{proof}

\begin{theorem}\label{q<pstratified}
Suppose $1\leq q < p < \infty$ and $\alpha>d/p$. Then  $M B_{\alpha}^{p,q} = M_{\alpha}^{p,q}$ with equivalence of norms.
\end{theorem}

\begin{proof}
Pick $m\in\N$ with $m>\alpha$. We shall prove that given $f\in M^{p,q}_\alpha$ and $g\in B^{p,q}_\alpha$
\begin{equation}\label{unodidue}
 \|fg  \|_p + \bigg( \sum_{k}   (2^{k\alpha}   \sup_{|y|<2^{-k}} \|\oG_{y}^{(m)}(fg)\|_{p}^q\bigg)^{1/q} \lesssim  \| f\|_{M^{p,q}_\alpha} \| g\|_{B^{p,q}_\alpha}.
\end{equation}
This implies that $\| f\|_{M^{p,q}_\alpha} \gtrsim \|f\|_{MB^{p,q}_\alpha}$.

By~\eqref{fgpstrat}
\[
\| f g\|_{p } \lesssim  \|g\|_{B_\alpha^{p,q}}  \|  f \|_{B_\alpha^{p,q}, {\unif}}\lesssim   \|g\|_{B_\alpha^{p,q}}  \|  f \|_{M^{p,q}_\alpha} ,
\]
the last inequality by Lemma~\ref{embeddingschain}, as well as  (we maintain the same notation as that of Proposition~\ref{prop_embed3})
\[
 \bigg( \sum_{k}   (2^{k\alpha}   \sup_{|y|<2^{-k}} \|\oG_{y}^{(m)}(fg)\|_{p}^q\bigg)^{1/q}  \lesssim \sum_{j=0}^{2m} \sigma_{j}.
\]
Since to estimate the terms with $j=0,\dots, m$ we did not use any condition on $p$ and $q$, we might argue in the same manner and get
\[
 \sum_{j=0}^{m} \sigma_{j} \lesssim  \|g\|_{B_\alpha^{p,q}}  \|  f \|_{B_\alpha^{p,q }, \unif} \lesssim \|g\|_{B_\alpha^{p,q}}  \|  f \|_{M^{p,q}_\alpha}
\]
again by Lemma~\ref{embeddingschain}.

We are left with considering the case $ j=m+1, \dots, 2m$. Select, by Lemma~\ref{covering}, $N$ disjoint families of indices $I_{k}$, $k=1,\dots, N$ with the property that
\[
\N =\bigcup_{k=1}^{N} I_{k}, \qquad  d_C(x_{\ell},x_{h}) \geq 2m+4 \quad \forall \ell,h \in I_{k}, \, \ell\neq k, \; \forall \, k.
\] 
Then, for $|y|\leq 1$ and $\ell,h \in I_{k}$, $\ell\neq h$,
\[
\supp \oG_{y}^{(m)}(\eta_{\ell}f)\cap  \supp \oG_{y}^{(m)}(\eta_{h}f) \subseteq  x_{\ell}B_{m+2} \cap x_{h}B_{m+2}  = \emptyset
\]
and thus
\begin{equation*}
\begin{split}
\sum_{n \in I_{k}}|\gamma_{n}|^{p}\|\oG_{y}^{(m)}(\eta_{n}f)\|_{p}^{p} 
&= \sum_{n \in I_{k}} \int_{G}|\gamma_{n}|^{p} |\oG_{y}^{(m)}(\eta_{n}f)|^{p} \, \dd\lambda \\
& = \int_{G}\Big| \sum_{n \in I_{k}} \oG_{y}^{(m)}(\gamma_{n}\eta_{n}f)\Big|^{p}\, \dd \lambda = \Big\| \sum_{n \in I_{k}} \oG_{y}^{(m)}(\gamma_{n}\eta_{n}f)\Big\|_{p}^{p}.
\end{split}
\end{equation*}
Now fix $k=1,\dots, N$ and pick the sequence (we assume  $g\neq 0$ here)
\[
\gamma_n = \gamma \frac{ \|\eta_{n} g\|_{\infty} }{ \| g\|_{B^{p,q}_{\alpha}}} \quad \mbox{if } n\in I_{k}, \qquad \gamma_{n}=0 \quad \mbox{otherwise}.
\]
If $\gamma$ is small enough, then the sequence $(\gamma_n)$ is in $\ell^{p}$ and has $\ell^{p}$ norm smaller than $1$; recall~\eqref{alphamenoeps}. This also shows that $\gamma$ can be chosen independent of $g$ and $k$.

Therefore, by~\eqref{j>m}
\begin{align*}
\sigma_j &\lesssim \bigg(\sum_{k}\bigg(2^{k\alpha p} \sup_{|y|< 2^{-k}}  
\sum_{n} \| \widetilde\xi_n g  \|_\infty   \|\oG_{y}^{(j)}(\eta_n f) \|_p^p 
\bigg)^{q/p} \bigg)^{1/q} \\
& \lesssim  \|g\|_{B^{p,q}_\alpha} \bigg(\sum_{k}\bigg(2^{k\alpha p} \sup_{|y| < 2^{-k}}  
\Big\| \sum_{n}   \oG^{(j)}_y(\gamma_{n}\eta_{n}f)\Big\|_{p}^p \bigg)^{q/p} \bigg)^{1/q}   \lesssim \|g\|_{B^{p,q}_\alpha}  \|f\|_{M^{p,q}_{\alpha}},
\end{align*}
which concludes the proof of~\eqref{unodidue}

Suppose now that $f \in  M B^{p,q}_\alpha$  and $(\gamma_n)\in \ell^{p}$. Observe that~\eqref{normafuori} still holds. By the algebra property of $B^{p,q}_{\alpha}$,
\begin{equation}\label{MB}
\Big\| \sum_{n\in I_{k}}\gamma_{n}\eta_{n}f\Big\|_{B^{p,q}_{\alpha}} \lesssim \|f\|_{MB^{p,q}_{\alpha}} \Big\| \sum_{n\in I_{k}}\gamma_{n}\eta_{n}\Big\|_{B^{p,q}_{\alpha}},
\end{equation}
where
\begin{align*}
\Big\| \sum_{n \in I_{k}} & \gamma_{n}\eta_{n}\Big\|_{B^{p,q}_{\alpha}} 
 \lesssim  \Big\|   \sum_{n \in I_{k}} \gamma_{n}\eta_{n}\Big\|_{p} + \bigg(\sum_{k} 2^{k\alpha q} \sup_{|y| < 2^{-k}}  \bigg(\sum_{n}
\|\oG_{y}^{(m)}( \gamma_{n}\eta_n)\|_{p}^p \bigg)^{q/p} \bigg)^{1/q} \\
&\lesssim \Big(   \sum_{n}\| \gamma_{n}\eta_{n}\|_{p}^p\Big)^{1/p} + \bigg(\sum_{k} 2^{k\alpha q} \sup_{|y| < 2^{-k}}  \bigg(\sum_{n}|\gamma_{n}|^{p} \|\oG_{y}^{(m)}\eta \|_{p}^p \bigg)^{q/p} \bigg)^{1/q} \\
& \lesssim  \|(\gamma_{n})\|_{\ell^{p}} \| \eta\|_{B^{p,q}_\alpha} \lesssim 1.
\end{align*}
By~\eqref{normafuori} and~\eqref{MB} we conclude $\|f\|_{M^{p,q}_{\alpha}} \lesssim  \|f\|_{MB^{p,q}_{\alpha}}$, and this completes the proof.
\end{proof}

\begin{remark}\label{remarkHOD} 
As already observed, the main obstacle to proving Theorem~\ref{teoB} for all $\alpha>0$, or equivalently Theorem~\ref{teoBstratified} beyond the stratified groups case, was for us the lack of a suitable notion of finite differences of order larger than $1$. It seems not clear, indeed, what the analogue of~\eqref{Gythetam} on a general Lie group should be.

The case of second order differences is somewhat special, and was actually considered by several authors (cf., e.g.,~\cite{Folland,Saka}) in different generalities. One may indeed define a \emph{symmetric} second-order difference of the form
\[
\oS^{(2)}_{y} f(x) = f(xy^{-1}) - 2f(x) + f(xy), \qquad x,y\in G.
\]
Since $\oS^{(2)}_{y} f  = \oD^{(2)}_{y} f( \cdot \, y^{-1})$, where
\[
\oD^{(2)}_{y} f(x) = f(xy^{-2}) - 2f(xy^{-1}) + f(x),
\]
it is tempting to define, for $m\in\mathbb N$ and $y\in G$, the finite difference of order $m$ as
\begin{equation}\label{Dm}
\oD_y^{(m)} f(x) = \sum_{\ell =0}^{m} (-1)^{m -\ell}   \binom{m}{\ell}   f(x y^{-\ell} ) .
 \end{equation}
When $G$ is a Euclidean space, this definition is nothing but the classical one, and when $m=1$ it is precisely~\eqref{oD}, i.e.\ $\oD^{(1)}_{y} = \oD_{y}$. However, if $G$ is stratified and $m\geq 2$, then~\eqref{Dm} does not coincide with~\eqref{Gythetam}, not even when $\theta=0$. Observe indeed that in general $x^{m} \neq \delta_{m}(x)$. Nevertheless, this differences do have remarkable properties, as a ``Leibniz rule'' in the spirit of~\eqref{leibnizG}, that is for all $m\in \N$
\[
\oD_y^{(m)}(f  g)(x) = \sum_{j=0}^{m} \binom{m}{j}   \oD_y^{(m-j)}f(xy^{-j})  \oD_y^{(j)}g(x),\qquad x,y\in G  .
\]
It is not clear to us whether $\oS_{y}^{(2)}$ and more generally $\oD_y^{(m)}$ satisfy the analog of Lemma~\ref{lemmaderstrat2}~(2), or the characterization~\eqref{Besovequivnorm} of the Besov norm, which are essential ingredients in our argument. This seems an interesting direction for future research in its own right.
\end{remark}

\end{document}